\documentclass[11pt]{article}

\usepackage{fullpage}
\usepackage{amsmath, amsthm, amssymb}
\usepackage{bbm}
\usepackage{authblk}
\usepackage{mathtools}
\usepackage{graphicx}
\usepackage{enumitem}
\usepackage{xcolor}
\usepackage{booktabs}
\usepackage{natbib}
\usepackage{subfig}
\usepackage{multirow}
\usepackage{float}
\usepackage{tikz}
\usepackage{longtable}
\usepackage{indentfirst}
\usepackage{setspace}
\usepackage{xurl}
\RequirePackage[colorlinks = true, linkcolor = blue, citecolor=blue,urlcolor=blue]{hyperref}

\usepackage{amsthm}

\newtheorem{theorem}{Theorem}
\newtheorem{prop}{Proposition}
\newtheorem{lemma}{Lemma}

\theoremstyle{definition}
\newtheorem{definition}{Definition}

\newtheorem{kernelexampleinner}{Example}

\newenvironment{kernelexample}
  {\begin{kernelexampleinner}}
  {\end{kernelexampleinner}}
\newtheorem{lossexampleinner}{Example}

\newenvironment{lossexample}
  {\begin{lossexampleinner}}
  {\end{lossexampleinner}}

\theoremstyle{remark}

\DeclareMathOperator*{\argmin}{argmin}
\DeclareMathOperator*{\diag}{diag}
\DeclareMathOperator{\MSE}{MSE}

\DeclareMathOperator{\Ran}{Ran}

\DeclareMathOperator{\RCV}{RCV}

\DeclareMathOperator{\Span}{Span}
\DeclareMathOperator{\sign}{sign}

\definecolor{ccol}{rgb}{.9,.35,0}

\definecolor{c1}{rgb}{0,  0, 0}
\definecolor{c2}{rgb}{0,  0, 0}
\definecolor{c3}{rgb}{128,  0, 128}
\definecolor{ccol}{rgb}{.9,.35,0}

\begin{document}

\title{Generalized nonparametric regression in reproducing kernel Hilbert spaces: Consistency and rates of convergence}
\author{Ioannis Kalogridis}
\affil{School of Mathematics and Statistics, University of Glasgow, United Kingdom}
\date{\today}
\maketitle
	
\begin{abstract}

We develop a comprehensive theory for regularized M-estimation in reproducing kernel Hilbert spaces. Under mild conditions we establish existence and measurability of the estimators, covering a wide range of convex and non-convex losses, including bounded robust losses. We further prove sharp rates of convergence with an explicit bias-variance decomposition governed by a novel complexity measure. We show that the variance is independent of misspecification, while the bias depends on a source condition parameter known in the learning literature. For tensor product Sobolev spaces we obtain new rates that connect to spaces of functions with dominating mixed smoothness, substantially extending existing results and explaining why these estimators circumvent the curse of dimensionality. Our methodology, combining elements from  both functional analysis and empirical process theory, allows for an asymptotic linearisation of the objective function that avoids both closed-form solutions and global Lipschitz assumptions, and may be of independent interest. The estimators are implemented in C++ and theory is supported by numerical experiments.

\end{abstract}



\section{Introduction}
Given independently sampled observations $\{(x_i,y_i)\}_{i=1}^n\subset X \times \mathbbm{R}$ where $X \subset \mathbbm{R}^d$, consider the  multivariate nonparametric regression model
\begin{align}
\label{eq:npm}
y_i = f_0(x_i) + \epsilon_i, \quad (i=1, \ldots, n),
\end{align}
where $f_0: \mathbbm{R}^d \to \mathbbm{R}$ is the unknown regression function and $\{\epsilon_i\}_{i=1}^n$ are independent and identically distributed additive error terms. When little is known a priori about $f_0$, a natural strategy is to estimate it by minimizing a regularized empirical risk over a reproducing kernel Hilbert space (RKHS) $\mathcal{H}$, that is,
\begin{align}
\label{eq:est}
\widehat{f}_n = \argmin_{f \in \mathcal{H}} \left[ \frac{1}{n} \sum_{i=1}^n \rho(y_i - f(x_i)) + \lambda \|f\|_{\mathcal{H}}^2 \right],
\end{align}
where $\rho: \mathbbm{R} \to \mathbbm{R}$ is a convex loss function and $\lambda > 0$ is a regularization parameter balancing fidelity to the data and smoothness. This formulation provides a broad estimation framework incorporating seamlessly multiple predictor variables. Moreover, the user has the freedom to specify the loss function depending on the characteristics of the data and goals of the analysis. Examples include not only robust loss functions, e.g., $\rho_k(x) = (x^2/2)1(|x| \leq k) + k(|x|-k/2)1(|x|>k)$ for $k>0$, aimed at reducing the influence for outlying observations on the estimates, but also asymmetric loss functions, such as the quantile loss $\rho_{\tau}(x) = x(\tau-1(x<0))$, aimed at extracting additional information from the sample. The estimator $\widehat{f}_n$ is called the regularized M-estimator and is the object of our study. We refer to the monographs \citep{Huber:2009,Mar:2019} for introductions and illustrative examples of M-estimation.

While the special case of \eqref{eq:est} with $\rho(x) = x^2/2$, the least-squares estimator, is by now a staple in the modern machine learning literature having been studied in numerous works, for example, in \citep{Cucker:2005, C:2007, St:2012, Fischer:2020, Zhang:2023}, the treatment of the general estimator \eqref{eq:est} has received significantly less attention. This disparity is surprising, as numerous practically important procedures, including quantile, expectile and robust regression, arise as special cases of \eqref{eq:est}. Moreover, these methods are often preferable to least-squares estimation either due to their robustness to outlying observations or their ability to estimate features of the conditional distribution beyond the mean, such as quantiles or expectiles.

The relative scarcity of theoretical results is largely explained by the technical difficulties posed by general loss functions. Unlike the least-squares estimator, regularized M-estimators do not admit explicit representations in terms of integral operators and kernel eigenfunctions, compare, e.g., with equation (21) in \citep{Fischer:2020}. As a consequence, many of the spectral techniques underlying modern RKHS theory are no longer directly applicable and one needs to argue in terms of the objective function in order to establish theoretical properties. To the best of our knowledge, no unifying framework has emerged and the literature remains fragmented with only a few particular combinations of losses and kernels having been investigated. Quantile regression in RKHS was studied by \citet{Li:2007}, who obtained rates of convergence under restrictive assumptions including boundedness of both the response variables and the RKHS. Subsequently, \citet{Eberts:2013} investigated least-squares and quantile regression in Gaussian RKHSs under Besov smoothness assumptions on the target function, while \citet{Far:2019} considered expectile regression, still in Gaussian RKHSs, and sharpened the rate obtained  by \citet{Eberts:2013}.

Rather than focusing on specific loss functions and RKHSs, this manuscript fills the gap by developing a theoretical framework for the analysis of \eqref{eq:est} in full generality. Our framework covers all commonly used convex absolutely continuous losses $\rho$ (fully worked out examples are given in Section~\ref{sec:loss}) and all RKHSs, even RKHSs on non-Euclidean domains such as spheres (see Section~\ref{sec:kernel} for specific examples). The main mathematical insight of this paper is that the variance of the regularized M-estimators is governed by a novel complexity measure $\mathcal{N}_{\infty}(\lambda)$, depending on $\lambda$, that we call \textit{spectral complexity}. The principle applies very broadly yielding explicit rates of convergence. We apply it, in particular, on the intriguing tensor product Sobolev space $\mathcal{H}^m([0,1])^{\otimes d}$ extending existing results and establishing the optimality of a large family of estimators. We summarise the main contributions of this work below.

\begin{enumerate}
\item We establish the theoretical existence and measurability of the regularized M-estimator $\widehat{f}_n$ under mild conditions: Proposition~\ref{prop:1} requires only that $\rho$ is continuous and bounded from below.

\item We prove consistency and sharp rates of convergence for the regularized M-estimator $\widehat{f}_n$ defined in \eqref{eq:est} with an explicit bias-variance decomposition (Theorem~\ref{thm:1}). In particular,
  \begin{align*}
\int_X \{\widehat{f}_n(x) - f_0(x)\}^2 v(dx) = O_\mathbb{P}\left( \frac{\mathcal{N}_\infty(\lambda)}{n} \right) + O_\mathbb{P}(\lambda^\beta),
 \end{align*}
where $v$ is the distribution of the $x_i$. The first term on the right represents the variance; it is governed by the spectral complexity $\mathcal{N}_\infty(\lambda)$ (introduced in Definition~\ref{def:1}) and is independent of misspecification. The second term represents the squared bias and depends on the source condition parameter $\beta \in (0,1]$ (introduced by \citet{St:2012}). Explicit bounds on $\mathcal{N}_\infty(\lambda)$ are derived for several important kernels, including Gaussian, (tensor and multivariate) Sobolev and zonal kernels. 

\item For tensor Sobolev spaces $\mathcal{H}^m([0,1])^{\otimes d}$ (defined and discussed in Section~\ref{sec:tensor}), we obtain new rates of convergence (Theorem~\ref{thm:2}) that hold for a broad family of loss functions and vastly extend the least-squares analysis of \citet{Lin:2000}. We also establish a direct connection with Sobolev spaces with dominating mixed smoothness widely used in approximation theory; this connection provides insight into why tensor-product-based estimators appear to sidestep the curse of dimensionality.

\item Our methodology, combining functional analysis and empirical processes, is novel and widely applicable. We linearise the objective function around the population counterpart of the estimator (the abstract regularized interpolant, see the appendix). The remainder of this linearisation is handled with symmetrisation and contraction. Our analysis, requiring neither smoothness of the loss function nor global Lipschitz behaviour, recovers and extends existing results to useful loss functions that have not been hitherto studied (popular examples are worked out in detail in Section~\ref{sec:loss}).

\end{enumerate}

Taken together, these contributions provide a flexible, robust, and theoretically sharp framework for nonparametric regression in RKHS that significantly broadens the scope of existing results while delivering explicit and optimal rates of convergence.

\vspace{0.3cm}
\textit{Notation}: In what follows, $\mathcal{L}_v^p(X)$,\ $p \in [1,\infty]$, will denote the space of $p$-integrable functions on a domain $X \subset \mathbbm{R}^d$ with respect to a nonnegative measure $v$. When $v$ is omitted, $\mathcal{L}^p(X)$ will refer to the space of $p$-integrable functions with respect to Lebesgue measure on $X$. $\mathcal{C}(X)$ will denote the space of continuous functions on $X$. The RKHS will be denoted with $\mathcal{H}$, the dependence on the kernel $K$ being implicit but clear from context.  Inner products and norms will be denoted with the traditional notation $\langle \cdot, \cdot \rangle$ and $\|\cdot\|$ respectively. When we need to emphasize the type of inner product or norm used, these will be accompanied with subscripts, e.g., $\langle \cdot, \cdot \rangle_{\mathcal{H}}$ for the standard inner product on $\mathcal{H}$.

\section{Framework and assumptions}

\subsection{Reproducing kernel Hilbert spaces and spectral complexity}
\label{sec:kernel}

Let $\mathcal{H}$ denote a reproducing kernel Hilbert space (RKHS) of real-valued functions on a domain $X \subset \mathbbm{R}^d$ with a continuous positive definite kernel $K: X \times X \to \mathbbm{R}$, that is, for all $N \in \mathbbm{N}$, all sets of pairwise distinct points $\{w_1, \ldots, w_N\} \subset X$ and all $a \in \mathbbm{R}^n$, 
\begin{align*}
\sum_{j=1}^N \sum_{k=1}^N a_j a_k K(w_j,w_k)>0.
\end{align*}
It is known that bounded, translation invariant kernels in $\mathcal{L}^1(\mathbbm{R}^d)$ are positive definite if, and only if, their Fourier transform is nonnegative and non-vanishing \citep[Theorem 6.11]{Wend:2005}; this result covers the commonly used Mat\'ern and Gaussian kernels. The two characteristic features of the RKHS $\mathcal{H}$ are (i) $x \mapsto K(x,y) \in \mathcal{H}$ for every $y \in X$ and (ii) $f(x) = \langle f, K(x,\cdot)\rangle_{\mathcal{H}}$. The reader is referred to \citep{S:2002, Ber:2004, Wend:2005, St:2008} for introductions to RKHS and learning theory.

Throughout, we shall impose the following assumption on $X$ and the distribution (push-forward measure) of the predictor variables $\{x_i\}_{i=1}^n$, $v$. We recall that a measure $v$ is supported on all of $X$ if, and only if, for every non-empty relatively open $O \subset X$, $v(O)>0$.

\begin{enumerate}[label=(A\arabic*), ref=(A\arabic*)]
\item \label{A1} $X$ is a compact subset of $\mathbbm{R}^d$ and the predictors $\{x_i\}_{i=1}^n$ are independent and identically distributed random vectors with distribution $v$ supported on all of $X$.
\end{enumerate}
\noindent
Assumption~\ref{A1} is a clean sufficient condition ensuring that the identity mapping $\mathcal{I}: \mathcal{H} \to \mathcal{L}^2_v(X)$ is injective, which is a necessary and sufficient condition for Mercer's theorem to hold. It is satisfied, for example, if $X = [0,1]^d$ and $v$ has Lebesgue density bounded away from zero and infinity, or, if $X = \mathbb{S}^{d-1}$ (a sphere in $\mathbbm{R}^d$) and $v$ is the canonical surface measure. This level of generality is highly desirable, as it allows for the modelling of non-Euclidean data.

With assumption \ref{A1}	 in place, Mercer's theorem, see, e.g., \citep[Theorem 12.20]{Wain:2019} or \citep[Corollary 3.5]{St:2012}, applies and the kernel $K$ may be expanded as
\begin{align}
\label{eq:merc}
K(x,y) = \sum_{j=1}^{\infty} \mu_j \phi_j(x) \phi_j(y), \quad (x,y) \in X \times X,
\end{align}
where $\{(\mu_j, \phi_j)\}_{j=1}^{\infty}$ are eigenvalue and eigenfunction pairs of the Hilbert-Schmidt integral operator $T_K : \mathcal{L}^2_v(X) \to \mathcal{L}^2_v(X)$ defined as $T_K(f)(x) = \int_{X} K(x,y) f(y) v(dy)$ for $f \in \mathcal{L}^2_v(X)$. The convergence of the series in \eqref{eq:merc} is absolute and uniform on $X \times X$. The positive definiteness of $K$ is inherited by $T_K$ \citep[see, e.g.][Theorem 4.6.4]{Hsing:2015}, so that $\mu_j>0$ and $\mu_j \to 0$ as $j \to \infty$, as Hilbert-Schmidt operators are compact with zero the only accumulation point for their eigenvalues. 

By the spectral theorem \citep[see, e.g.,][Theorem 4.2.4]{Hsing:2015} the eigenfunctions $\{\phi_j\}_{j=1}^{\infty}$ provide a complete orthonormal system (a basis) for $\mathcal{L}^2_v(X)$. This observation combined with Mercer's theorem \eqref{eq:merc} allows the characterization of $\mathcal{H}$ as
\begin{align}
\label{eq:rkhs}
\mathcal{H} = \left\{f \in \mathcal{L}^2_v(X): f = \sum_{j=1}^{\infty} f_j \phi_j , \ \sum_{j=1}^\infty \frac{f_j^2}{\mu_j} < \infty \right\}.
\end{align}
That is, $\mathcal{H} \subset \mathcal{L}^2_v(X)$ consists  of $\mathcal{L}^2_v(X)$-functions whose Fourier series converges sufficiently fast (recall that $\mu_j \to 0$). Moreover, the inner product of the RKHS $\mathcal{H}$ is given by
\begin{align*}
\langle f, g \rangle_{\mathcal{H}} = \sum_{j=1}^{\infty} \frac{f_j g_j}{\mu_j}, \quad f,g \in \mathcal{H}.
\end{align*}
The associated norm is $\|f\|_{\mathcal{H}} := \sqrt{\langle f, f \rangle_{\mathcal{H}}}$. With these definitions, it can be seen that $\langle \phi_i, \phi_j \rangle_{\mathcal{L}^2_v(X)} = \delta_{ij}$ and $\langle \phi_i, \phi_j \rangle_{\mathcal{H}} = \delta_{ij} \mu_i^{-1}$ where $\delta_{ij}$ denotes Kronecker's delta, i.e., $\delta_{ij} = 1$ for $i = j$ and $\delta_{ij} =0$ otherwise. Moreover, $\{\sqrt{\mu_j}\phi_j\}_{j=1}^{\infty}$ is a complete orthonormal system for $\mathcal{H}$. The interested reader is referred to \citep{St:2012} for insightful discussions and counterexamples.

Before we investigate rates of convergence for the regularized M-estimator $\widehat{f}_n$ defined in \eqref{eq:est}, we settle the basic questions of existence and measurability; Proposition~\ref{prop:1} covers all commonly used loss functions.

\begin{prop}
\label{prop:1}
Suppose that \ref{A1} is satisfied, $\rho: \mathbbm{R} \to \mathbbm{R}$ is continuous and bounded from below, i.e., $\inf_{x \in \mathbbm{R}} \rho(x) \geq C$ for some $C \in \mathbbm{R}$, and $\lambda>0$. Then, $\widehat{f}_n$ is well-defined as an element of $\mathcal{H}$ and measurable. 
\end{prop}

\begin{proof}
Let $L_n: \mathcal{H} \to \mathbbm{R}$ denote the objective function, i.e., 
\begin{align*}
L_n(f)= \frac{1}{n} \sum_{i=1}^n \rho \left( y_i -f(x_i) \right) + \lambda \|f\|_{\mathcal{H}}^2.
\end{align*}
The infimum $\inf_{f \in \mathcal{H}} L_n(f)$ is finite, as $L_n$ is bounded from below by $C$. Let $\{f_k\}_k$ be any minimizing sequence such that $L_n(f_k) \to \inf_{f \in \mathcal{H}} L_n(f)$, as $k \to \infty$. The minimizing sequence $\{f_k\}_k$ is bounded, for, if that were not true and there was a subsequence $\{f_{k_m}\}_m$ such that $\|f_{k_m}\|_{\mathcal{H}} \to \infty$, then $L_n(f_{k_m}) \to \infty$ and at the same time $L_n(f_{k_m}) \to \inf_{f \in \mathcal{H}} L_n(f)$. This is impossible, as the infimum is finite. Hilbert spaces are reflexive, hence, given that $\{f_k\}_k$ is bounded, there exists a weakly convergent subsequence, which, after relabelling, we still denote with $\{f_k\}_k$. Thus, $f_k \rightharpoonup g$ for some $g \in \mathcal{H}$. The norm is weakly lower semicontinuous with respect to weak convergence, so
$\|g\|_{\mathcal{H}}^2 \leq \liminf_{k \to \infty} \|f_k\|_{\mathcal{H}}^2$.
Moreover, by the reproducing property and weak convergence $f_k(x_i) = \langle f_k, K(x_i,\cdot) \rangle_{\mathcal{H}}  \to \langle g, K(x_i,\cdot) \rangle_{\mathcal{H}} = g(x_i)$ for all $i = 1, \ldots, n$. The continuity of $\rho$ now implies that $L_n$ is weakly lower semicontinuous and
\begin{align*}
\inf_{f \in \mathcal{H}}L_n(f) \leq L_n(g) \leq \liminf_{k \to \infty} L_n(f_k) = \inf_{f \in \mathcal{H}}L_n(f),
\end{align*}
showing that the infimum is attained in $\mathcal{H}$. Call the argmin $\widehat{f}_n$.

The representer theorem \citep[Theorem 5.5]{St:2008} ensures that $\widehat{f}_n \in V_n := \Span\{K(x_i, \cdot),\ i=1, \ldots, n\}$. Hence, we parameterise $f \in V_n$ by $\alpha \in \mathbbm{R}^n$ through $f_{\alpha}(\cdot) = \sum_{i=1}^n \alpha_i K(\cdot, x_i)$ and the minimisation is carried out over $\mathbbm{R}^n$. The resulting objective function $\widetilde{L}_n(\omega, \alpha)$ defined as
\begin{align*}
\widetilde{L}_n(\omega,\alpha) = \frac{1}{n} \sum_{i=1}^n \rho\left( y_{i}(\omega)-f_{\alpha}(x_i(\omega))) \right) + \lambda \alpha^{\top} K \alpha
\end{align*}
is a Carath\'eodory function: $\omega \mapsto L_n(\omega,\alpha)$ is measurable for every $\alpha \in \mathbbm{R}^n$ and $\alpha \mapsto L_n(\omega,\alpha)$ is continuous for every $\omega \in \Omega$. In the terminology of \citep[Example 14.29][]{Rock:1998}, $\widetilde{L}_n$ is a normal integrand. Hence, by Theorem 14.37 of these authors, the argmin $\widehat{f}_n$ is measurable.
\end{proof}

It is worth noting that, for convex $\rho$, the strict convexity of $f \mapsto \|f\|_{\mathcal{H}}^2$ (recall that Hilbert spaces are uniformly convex) makes $L_n$ strictly convex on $\mathcal{H}$. Therefore, $\widehat{f}_n$ is unique even when $\rho$ is convex but not strictly convex, e.g., the quantile loss; this need not be the case for general penalty functionals possessing non-trivial null spaces arising, e.g., in smoothing spline estimation. But beyond convex losses, the proposition applies to very robust bounded loss functions, such as the Tukey or Hampel losses, guaranteeing the existence of a measurable minimizer, though in this case uniqueness for finite samples cannot be assured.

Since our goal is to derive rates of convergence with respect to the $\mathcal{L}^2_v(X)$-norm it is convenient to equip $\mathcal{H}$ with a different inner product that captures the effect of the penalization. The inner product most suited for our analysis is given by
\begin{align}
\label{eq:ip}
\langle f,g \rangle_{\mathcal{H},\lambda} := \langle f,g \rangle_{\mathcal{L}^2_v(X)} + \lambda \langle f, g \rangle_{\mathcal{H}}, \quad f,g \in \mathcal{H},
\end{align}
for each $\lambda>0$. The associated norm is $\|f\|_{\mathcal{H},\lambda} := \sqrt{\langle f, f \rangle_{\mathcal{H},\lambda}}$. Thus, $\langle \cdot, \cdot \rangle_{\mathcal{H},\lambda}$ and $\|\cdot\|_{\mathcal{H},\lambda}$ depend on $n$, via $\lambda$. While at a first glance the norm $\|\cdot\|_{\mathcal{H},\lambda}$ may appear to induce a different topology on $\mathcal{H}$ relative to $\|\cdot\|_{\mathcal{H}}$, this is not the case, as these norms are equivalent. Specifically, on the one hand, it is clear that $\lambda^{1/2}\|f\|_{\mathcal{H}} \leq  \|f\|_{\mathcal{H},\lambda}$ and, on the other hand, $\mathcal{H}$ embeds continuously (even compactly) into $\mathcal{L}_v^2(X)$ \citep[see, e.g.,][Lemma 2.6]{St:2012} and so there exists $c_0>0$ such that $\|f\|_{\mathcal{L}^2_v(X)}^2 \leq c_0 \|f\|_{\mathcal{H}}^2$ for all $f \in \mathcal{H}$.  Combining these two bounds, $\lambda^{1/2} \|f\|_{\mathcal{H}} \leq \|f\|_{\mathcal{H},\lambda} \leq (c_0+\lambda)^{1/2} \|f\|_{\mathcal{H}}$, for all $f \in \mathcal{H}$, establishing the equivalence of the norms and the induced topologies.

The equivalence of the norms $\|\cdot\|_{\mathcal{H}}$ and $\|\cdot\|_{\mathcal{H},\lambda}$ implies that $\mathcal{H}$ equipped with $\langle \cdot, \cdot \rangle_{\mathcal{H},\lambda}$ is an RKHS on its own right (the evaluation functionals are still continuous) and thus there exists a unique kernel $K_{\lambda}: X \times X \to \mathbbm{R}$ (depending on $\lambda$) such that $f(x) = \langle K_{\lambda}(x,\cdot), f \rangle_{\mathcal{H},\lambda}$ for all $f \in \mathcal{H}$. Recalling that the eigenfunctions $\{\phi_j\}_{j=1}^{\infty}$ appearing in \eqref{eq:merc} ``diagonalize" both $\langle \cdot, \cdot \rangle_{\mathcal{L}^2_v(X)}$ and $\langle \cdot, \cdot \rangle_{\mathcal{H}}$, it is straightforward to verify that $K_{\lambda}$ may be written as
\begin{align}
\label{eq:RK}
K_{\lambda}(x,y) = \sum_{j=1}^{\infty} \frac{\phi_j(x) \phi_j(y)}{1+\lambda \gamma_j},
\end{align}
where $\gamma_j = 1/\mu_j$ and for each $\lambda>0$ the series converges, similarly to \eqref{eq:merc}, absolutely and uniformly. In particular, $K_{\lambda}$, as the uniform limit of continuous functions, is continuous on $X \times X$.

Continuity of $K_{\lambda}$ on the compact $X$ implies the compact embedding $\mathcal{H} \hookrightarrow \mathcal{C}(X)$, see, e.g., \citep[Proposition 4.30]{St:2008}. The smallest possible embedding constant is critical in our analysis and we propose to call it \textit{spectral complexity}.

\begin{definition}[Spectral complexity]
\label{def:1} For each $\lambda \in (0,\infty)$, the spectral complexity $\mathcal{N}_{\infty}(\lambda) \in (0,\infty)$ is defined as
\begin{align*}
\mathcal{N}_{\infty}(\lambda) := \sup_{f \in \mathcal{H}, f \neq 0} \frac{\|f\|^2_{\infty}}{\|f\|_{\mathcal{H},\lambda}^2}.
\end{align*}
\end{definition}
\noindent
The spectral complexity is not to be confused with the effective dimension widely used in the learning literature, e.g., in \citep{Fischer:2020}, and typically denoted with $\mathcal{N}(\lambda)$; in our notation, $\mathcal{N}(\lambda) = \sum_{j=1}^{\infty} (1+\lambda \gamma_j)^{-1}$. To clarify the distinction,  let us observe that, using \eqref{eq:RK}, by monotone convergence and the $\mathcal{L}^2_v(X)$-orthonormality of $\{\phi_j\}_{j=1}^{\infty}$, we have $\mathcal{N}(\lambda) = \int_{X} K_{\lambda}(x,x) v(dx)$. By the reproducing property and the Schwarz inequality, we obtain $\|f\|_{\mathcal{L}^2_v(X)}^2 = \int_{X} |\langle f, K_{\lambda}(x,\cdot) \rangle_{\mathcal{H},\lambda}|^2 v(dx) \leq \mathcal{N}(\lambda) \|f\|_{\mathcal{H},\lambda}^2$, for every $f \in \mathcal{H}$, as $K_{\lambda}(x,x) = \|K_{\lambda}(x,\cdot)\|_{\mathcal{H},\lambda}^2$. By contrast, Definition \ref{def:1} gives $\|f\|_{\infty}^2\leq \mathcal{N}_{\infty}(\lambda) \|f\|_{\mathcal{H},\lambda}^2$, for every $f \in \mathcal{H}$. Thus, while $\mathcal{N}(\lambda)$ yields $\mathcal{L}_v^2(X)$-control, $\mathcal{N}_{\infty}(\lambda)$ yields uniform control in the unit ball $\{f \in \mathcal{H}: \|f\|_{\mathcal{H},\lambda} \leq 1 \}$. This distinction is essential for the analysis of general convex losses whose subgradient may not be a linear function, but only a locally linear one at the population level, see assumption \ref{A4} and the proof of Theorem~\ref{thm:1} below. From this perspective, the spectral complexity $\mathcal{N}_{\infty}(\lambda)$ is the natural complexity measure for the class of estimators considered in this paper.

It may be verified that $\mathcal{N}_{\infty}(\lambda)$ is well-defined and finite for each $\lambda>0$; in particular, the following crucial relation holds.
\begin{prop}
\label{prop:2}
For each $\lambda \in (0,\infty)$, the spectral complexity is well-defined and
\begin{align*}
\mathcal{N}_{\infty}(\lambda) =  \sup_{x \in X} K_{\lambda}(x,x) < \infty.
\end{align*}
\end{prop}
\begin{proof}

By homogeneity, the result of the proposition is equivalent to $\mathcal{N}_{\infty}(\lambda) = \sup_{\|f\|_{\mathcal{H},\lambda} = 1} \|f\|_{\infty}^2 = \sup_{x \in X} K_{\lambda}(x,x)$. The reproducing property entails $|f(x)|^2 = |\langle f, K_{\lambda}(x,\cdot)\rangle_{\mathcal{H},\lambda}|^2$.  Interchanging suprema and using the dual norm on $\mathcal{H}$ equipped with $ \langle \cdot, \cdot \rangle_{\mathcal{H},\lambda}$, we obtain
\begin{align*}
\sup_{\|f\|_{\mathcal{H},\lambda} = 1} \sup_{x \in X} |f(x)|^2 =  \sup_{x \in X} \sup_{\|f\|_{\mathcal{H},\lambda} = 1} |\langle f, K_{\lambda}(x,\cdot) \rangle_{\mathcal{H},\lambda}|^2 =  \sup_{x \in X} \|K_{\lambda}(x,\cdot)\|_{\mathcal{H}, \lambda}^2,
\end{align*}
as $\sup_{\|f\|_{\mathcal{H},\lambda} = 1} |\langle f, K_{\lambda}(x,\cdot) \rangle_{\mathcal{H},\lambda}|^2 = \| K_{\lambda}(x,\cdot)\|_{\mathcal{H},\lambda}^2$. But, by the symmetry of $K_{\lambda}$ and the reproducing property, $\langle K_{\lambda}(x,\cdot), K_{\lambda}(x,\cdot) \rangle_{\mathcal{H},\lambda} = K_{\lambda}(x,x)$, hence
\begin{align*}
\sup_{\|f\|_{\mathcal{H},\lambda} = 1} \|f\|_{\infty}^2 = \sup_{x \in X} K_{\lambda}(x,x).
\end{align*}
Therefore, $\mathcal{N}_{\infty}(\lambda) = \sup_{x \in X} K_{\lambda}(x,x)$. For $\lambda>0$ the supremum is finite, as $x \mapsto K_{\lambda}(x,x)$ is continuous and $X$ is compact. 
\end{proof}

Inspection of our main results in Theorems \ref{thm:1} and \ref{thm:2} below reveals the second reason for the importance of spectral complexity: it determines the asymptotic variance, even under misspecification. Thus, in order to obtain sharp, explicit rates of convergence, we need to be able to bound $\mathcal{N}_{\infty}(\lambda)$ effectively. While this may be a difficult task, in general, the examples below show that we can often determine the behaviour of $\mathcal{N}_{\infty}(\lambda)$ as $\lambda \to 0$, which is all that is required for asymptotic rates of convergence. Before we give specific examples, let us note that if there exists an $\alpha \in (0,1)$ such that $\sup_{x \in X}\sum_{j=1}^{\infty} \mu_j^{\alpha} \phi_j^2(x)<\infty$ (under our assumptions this is always the case when $\alpha=1$) then
\begin{align*}
\mathcal{N}_{\infty}(\lambda) = \sup_{x \in X} \sum_{j=1}^\infty \frac{\mu_j^{1-\alpha}}{\mu_j+\lambda} \mu_j^{\alpha} \phi_j^2(x) \leq \sup_{t>0} \frac{t^{1-\alpha}}{t+\lambda} \sup_{x \in X} \sum_{j=1}^{\infty} \mu_j^{\alpha} \phi_j^2(x) = O(\lambda^{-\alpha}),
\end{align*}
but this polynomial bound may be loose when the eigenvalues decay very fast.

\begin{kernelexample}[Kernels with uniformly bounded eigenfunctions] Suppose that the integral operator $T_K$ possesses uniformly bounded eigenfunctions, that is, $\sup_j \|\phi_j\|_{\infty} \leq M$. Then, by Proposition \ref{prop:2}, $\mathcal{N}_{\infty}(\lambda) \leq M^2\sum_{j=1}^{\infty} (1+\lambda \gamma_j)^{-1}$. If for large $j$, $\mu_j =O(j^{-\alpha})$ for some $\alpha>1$ then, by integral approximation, $\mathcal{N}_{\infty}(\lambda) = O(\lambda^{-1/\alpha})$. If for large $j$, $\mu_j =O( e^{-c j^{1/d}})$ for some $c>0$ and $d \geq 1$ then $\mathcal{N}_{\infty}(\lambda) = O([\log(1/\lambda)]^d)$ as $\lambda \to 0$. Examples of kernels with uniformly bounded eigenfunctions are the Sobolev kernel on $[0,1]$ with $v$ Lebesgue measure, see, e.g., \citep[p. 60]{Hsing:2015} and the references therein, and translation invariant kernels on $[0,1]$ based on 1-periodic functions with $v$ again Lebesgue measure, see \citep[pp. 397--398]{Wain:2019}
\end{kernelexample}

\begin{kernelexample}[Gaussian kernel] Suppose $X = [0,1]^d$, $K(x,y) = e^{-\|x-y\|_{\mathbbm{R}^d}^2/h}$ for $h>0$ and $v$ has Lebesgue density bounded away from zero and infinity. It is known that $\mu_j = O( e^{-c j^{1/d}})$ and $\|\phi_j\|_{\infty} \leq b j^2$ for some $b,c>0$ (see Theorem 7 of \citep{Dommel:2025}). Therefore, by Proposition~\ref{prop:2}, $\mathcal{N}_{\infty}(\lambda) = O([\log(1/\lambda)]^{5d})$, as $\lambda \to 0$.
\end{kernelexample}

\begin{kernelexample}[Zonal kernels] Let $X = \mathbb{S}^{d-1}$ with $d \geq 2$ and consider the kernel $K(x,y) = (2-\langle x, y \rangle_{\mathbbm{R}^d})^{-1}$ for $x,y \in \mathbb{S}^{d-1}$. It is known, see, e.g., Table 1 in \citep{S:2021}, that $\mu_j = O(2^{-j})$. The eigenfunctions are given by the spherical harmonics and the addition formula \citep[Lemma 17.3]{Wend:2005} implies the existence of a $c_0>0$ such that
\begin{align*}
\mathcal{N}_{\infty}(\lambda) = \sup_{x \in \mathbb{S}^{d-1}} K_{\lambda}(x,x) \leq c_0 \sum_{j=1}^{\infty} \frac{j^{d-2}}{1+2^j\lambda} = O \left( \left[\log\left( \frac{1}{\lambda} \right)\right]^{d-1} \right),
\end{align*}
as $\lambda \to 0$.
\end{kernelexample}

\begin{kernelexample}[Sobolev kernels on $\mathbbm{R}^d$]

Sobolev kernels on $X \subset \mathbbm{R}^d$ with $d>1$, in general, do not possess uniformly bounded eigenfunctions, but we can still estimate $\mathcal{N}_{\infty}(\lambda)$ as follows. Assume that $v$ has Lebesgue density bounded away from zero and infinity and recall that the Sobolev space $\mathcal{H}^{m}(X)$ consists of functions in $\mathcal{L}^2(X)$ having all weak derivatives of order $\leq m$ in $\mathcal{L}^2(X)$. It is well-known that $\mathcal{H}^m(X)$ is an RKHS if, and only if, $2m>d$. Moreover, up to equivalent norms, $\mathcal{H}^m(X)$ is the RKHS induced by the restriction of the Mat\'ern kernel with parameter $v= m-d/2$ to $X$ \citep[see, e.g.][Corollary 10.48]{Wend:2005}. Suppose $\lambda \in (0,1]$. By Sobolev's inequality \citep[Theorem 3.9]{Agmon:1965} (taking $r = \lambda^{-1} \geq 1$ in Agmon's theorem), there exists a universal $c_0>0$ such that, for every $f \in \mathcal{H}^{m}(X)$,
\begin{align*}
\|f\|_{\infty}^2 \leq c_0 \lambda^{-d/(2m)}(\|f\|_{\mathcal{L}^2_v(X)}^2+\lambda \|f\|_{\mathcal{H}^m(X)}^2) = c_0 \lambda^{-d/(2m)} \|f\|_{\mathcal{H}^m(X), \lambda}^2.
\end{align*}
For $f \neq 0$, dividing and taking the supremum yields $\mathcal{N}_{\infty}(\lambda) = O(\lambda^{-d/(2m)})$ for all $\lambda \in (0,1]$.
\end{kernelexample}

Example K1 shows that in the presence of uniformly bounded eigenfunctions the spectral complexity $\mathcal{N}_{\infty}(\lambda)$ behaves like the effective dimension $\mathcal{N}(\lambda)$, as $\lambda \to 0$. But, as examples K2--K4 demonstrate, uniformly bounded eigenfunctions are not necessary for our analysis and estimates of $\mathcal{N}_{\infty}(\lambda)$ as $\lambda \to 0$ can be obtained even when $\sup_{j} \|\phi_j\|_{\infty} = \infty$. To obtain such estimates one may use Proposition~\ref{prop:2} (Examples K2--K3) or the very definition of $\mathcal{N}_{\infty}(\lambda)$ (Example K4), whichever is more practical in any given setting. 

\subsection{Assumptions on the loss function}
\label{sec:loss}

We now present and discuss the assumptions imposed on the loss function $\rho$. Our assumptions comprise two weak regularity conditions and one Fisher consistency condition that is widely satisfied, as we show by means of several examples.

\begin{enumerate}[label=(A\arabic*), ref=(A\arabic*)]
\setcounter{enumi}{1}
\item \label{A2} The loss function $\rho: \mathbbm{R} \to \mathbbm{R}$ is convex and absolutely continuous with derivative $\psi$ existing almost everywhere.
\item \label{A3} There exist constants $\kappa, M_1>0$ such that for all $|y| \leq \kappa$
\begin{align*}
|\psi(x+y)-\psi(x)| \leq M_1,
\end{align*}
uniformly in $x \in \mathbbm{R}$.
\item \label{A4} The errors $\{\epsilon_i\}_{i=1}^n$ are i.i.d., independent of the predictor vectors $\{x_i\}_{i=1}^n$, satisfy $\mathbb{E}[\psi(\epsilon_1)] =0$, $\mathbb{E}[|\psi(\epsilon_1)|^2]<\infty$ and there exists a $\gamma>0$ such that
\begin{align*}
\mathbb{E} \left[ \psi(\epsilon_1+t) \right] = \gamma t + o(t),
\end{align*}
as $t \to 0$.
\end{enumerate}

Assumption \ref{A2} ensures that the fundamental theorem of calculus holds in the Lebesgue sense, i.e., that we can write $\rho(x) = \rho(0) + \int_{0}^x \psi(t) dt$. This representation is used repeatedly in our proofs. The assumption is fulfilled  by all commonly used loss functions, including loss functions that are not continuously differentiable, such as the quantile loss. Assumption \ref{A3} is due to \citet{Bai:1994} and has been referred to as \textit{uniformly bounded local increments}. It is similarly a general assumption that is fulfilled by all commonly used losses, sufficient conditions including either uniform continuity or merely boundedness of $\psi$. For example, it is fulfilled by both $\psi(x) = x$  and $\psi(x) = \sign(x)$.
 
Assumption \ref{A4} is placed jointly on $\psi$ and the distribution of the errors $\{\epsilon_i\}_{i=1}^n$. The first part, namely the independence of errors and predictors, is standard in regression analysis. The moment conditions $\mathbb{E}[\psi(\epsilon_1)] = 0$ and $\mathbb{E}[|\psi(\epsilon_1)|^2 < \infty$ are widely used in M-estimation, see, e.g., \citep[Section 10.6]{Mar:2019}, and generalize the classical Gauss--Markov conditions. For bounded score functions $\psi$, the second moment condition is automatically satisfied whereas a sufficient (but by no means necessary) condition ensuring $\mathbb{E}[\psi(\epsilon_1)] = 0$ is that $\psi$ is bounded and odd and the distribution of $\epsilon_1$ is symmetric about zero, e.g., any $t$-distribution with positive degrees of freedom.

The second part of assumption \ref{A4} requires that $g(t) := \mathbb{E}[\psi(\epsilon_1+t)]$ is differentiable at $0$ with strictly positive derivative $\gamma$. This is a necessary condition for the minimizer to be well-separated in the limit and can be shown to hold for a wide variety of loss functions under appropriate, loss-specific, conditions on the error. In particular, it is trivial for $\rho(x) = x^2/2$ provided that $\mathbb{E}[\epsilon_1]=0$ and $\mathbb{E}[\epsilon_1^2] < \infty$. We now give several other interesting examples; throughout, we use $F$ to denote the cumulative distribution function of the error $\epsilon_1$, i.e., $F(t) = \Pr(\epsilon_1 \leq t)$.

\begin{lossexample}[Absolute loss and quantiles]
Consider first the absolute loss $\rho(x) = |x|$. We assume (i) $F$ has median at $0$, i.e., $F(0)=1/2$, and (ii) $F$ has Lebesgue density $f$ in a neighbourhood about $0$ with $f(0)>0$. Then, direct calculation yields
\begin{align*}
\mathbbm{E}[\psi(\epsilon_1+t)] = -\int_{x+t <0} F(dx) + \int_{x+t \geq 0}  F(dx) = 1-2F(-t) = 2f(0)t + o(t),
\end{align*}
as $t \to 0$. Thus, \ref{A4} is satisfied with $\gamma = 2 f(0)>0$. More generally, for the quantile loss $\rho_{\tau}(x) = x(\tau-1(x<0))$ with $\tau \in (0,1)$ we have $\mathbb{E}[\psi_{\tau}(\epsilon+t)] = \tau - F(-t-)$ where $F(x-)$ denotes the limit from the left at $x$. Therefore, \ref{A4} is satisfied with $\gamma = f(0)$ provided that $\epsilon_1$ has CDF $F$ satisfying (i) $F(0) = \tau$ and (ii) $F$ has Lebesgue density $f$ in a neighbourhood about $0$ satisfying $f(0)>0$.
\end{lossexample}

\begin{lossexample}[Vapnik's $\delta$-insensitive loss]
Fix $\delta \geq 0$ and consider the convex loss function $\rho_{\delta}(x) = \max(|x|-\delta,0)$. We assume (i) $F(\delta) + F(-\delta) = 1$ (ii) $F$ has Lebesgue density $f$ in a neighborhood of $-\delta$ and $\delta$ and (iii) $f(-\delta)+f(\delta)>0$. Then,
\begin{align*}
\mathbbm{E}[\psi_{\delta}(\epsilon_1+t)] = \int_{|x+t| \geq \delta} \sign(x+t) F(dx) = 1-F(\delta-t) - F(-\delta-t) = \{f(\delta)+f(-\delta)\} t + o(t),
\end{align*}
as $t \to 0$. Hence, \ref{A4} holds with $\gamma = f(\delta) + f(-\delta)>0$. For $\delta = 0$, $\rho_{\delta}(x)$ reduces to $\rho(x) = |x|$ and our assumptions reduce to $f(0)>0$, as in the previous example.
\end{lossexample}

\begin{lossexample}[Huber loss]
Fix $k>0$ and consider the Huber loss  $\rho_k(x) = (x^2/2) 1(|x| \leq k) + k(|x| - k/2) 1(|x|>k)$ blending absolute and square losses. We assume (i) $F$ is symmetric about $0$ and (ii) $F$ is continuous at $\{-k,k\}$. Then, integration by parts yields
\begin{align*}
\mathbb{E}[\psi_k(\epsilon_1+t) ] & = \int_{|x+t|\leq k} (x+t) F(dx) + k\int_{|x+t|> k} \sign(x+t) F(dx) 
\\ & = \int_{-k-t}^{k-t} x F(dx) + t\{F(k-t) -F(-k-t)\} + k(1-F(k-t))-kF(-k-t)
\\ & = -\int_{-k-t}^{k-t} F(x) d x + k = (2F(k)-1)t + o(t),
\end{align*}
as $t \to 0$. Note that $2F(k)>1$, because $k>0$ and $F(0)=1/2$. Thus, \ref{A4} is satisfied with $\gamma = 2F(k)-1$.
\end{lossexample}

\begin{lossexample}[Expectile loss] Fix $\alpha \in (0,1)$ and consider the expectile loss $\rho_{\alpha}(x) = x^2/2|\alpha-1(x < 0)|$. It is clear that $\rho_{\alpha}$ is continuously differentiable with derivative $\psi_{\alpha}(x) = \alpha x 1(x>0) + x(1-\alpha) 1(x \leq 0)$. We assume (i) the error has $\alpha$-expectile equal to zero, i.e., $\mathbb{E}[\psi_{\alpha}(\epsilon_1)] = 0$, (ii) $\mathbb{E}[\epsilon_1^2] < \infty$ and (iii) $F$ is continuous at $0$. Under these assumptions we obtain
\begin{align*}
\mathbb{E}[\psi_{\alpha}(\epsilon_1 +t)] & = \alpha \int_{-t}^{\infty} x F(dx) + \alpha t(1-F(-t)) + (1-\alpha) \int_{-\infty}^{-t} x F(dx) + (1-\alpha) t F(-t)
\\ & = \alpha \left\{ \int_{-t}^{\infty} x F(dx) - \int_{0}^{\infty} x F(dx) \right\} + (1-\alpha) \left\{\int_{-\infty}^{-t} x F(dx) - \int_{-\infty}^{0} x F(dx) \right\}
\\ & \quad + \alpha t (1-F(0)) + (1-\alpha) tF(0) + o(t) = \{\alpha+(1-2\alpha)F(0)\}t+o(t),
\end{align*}
as $t \to 0$, by the continuity of $F$ at $0$. Observe that for $\alpha>0$ the term inside the curly brackets is strictly positive and so assumption \ref{A4} is satisfied with $\gamma = \alpha + (1-2 \alpha)F(0)$.
\end{lossexample}

The reader will have noticed the underlying principle in the previous examples: our assumptions are general enough to permit score functions with a finite number of discontinuities so long as $F$ is absolutely continuous in a neighbourhood about these discontinuities. Score functions $\psi$ that are continuous but not everywhere differentiable, e.g., Huber and expectile losses, are easier to deal with requiring essentially that $F$ is continuous at the points of non-differentiability of $\psi$, i.e., these points are not atoms of $F$. In both cases, smoothness can be traded between $\psi$ and $F$ allowing for our conditions to be satisfied very broadly.

Assumption~\ref{A4} and our proofs can be adapted to cover independent but non-identically distributed $\{\epsilon_i\}_{i=1}^n$ at the cost of more cumbersome notation.  In that case, assumption~\ref{A4} would read $\mathbb{E} [\psi(\epsilon_i)] = 0$, $ \ \sup_n \max_{i \leq n} \mathbb{E}[|\psi(\epsilon_i)|^2] < \infty$ and 
\begin{align*}
\sup_n \max_{i \leq n} \left| \mathbb{E}[\psi(\epsilon_i+t)] - \gamma_i t\right| = o(t),
\end{align*}
as $t \to 0$, with $\{\gamma_i\}_{i=1}^n$ satisfying $0<\inf_n \min_{i \leq n} \gamma_i \leq \sup_n \max_{i \leq n} \gamma_i < \infty$. 

\section{Main results for general RKHS}
\label{sec:rkhs}

We aim to present bounds for the $\mathcal{L}^2_v(X)$-distance between $\widehat{f}_n$ and $f_0$ both in the case $f_0 \in \mathcal{H}$ and $f_0 \notin \mathcal{H}$. We are primarily interested in the case in which $f_0$ is less smooth than functions from $\mathcal{H}$ (the smoothness of these functions being determined by the smoothness of $K$, see, e.g., \citep[Theorem 10.45]{Wend:2005}). For this reason, we will assume that $f_0$ lies in an intermediate space with functions that are smoother than those in $\mathcal{L}^2_v(X)$ (which are technically not even functions) but no smoother than those in $\mathcal{H}$. In particular, and in line with previous works \citep{St:2012,Fischer:2020, Zhang:2023}, we assume that $f_0$ lies in the $\beta$-power space of functions $[\mathcal{H}]^{\beta}$.

\begin{enumerate} [label=(A\arabic*), ref=(A\arabic*)]
\setcounter{enumi}{4}
\item \label{A5} There exists a $\beta \in (0,1]$ such that
\begin{align*}
f_0 \in [\mathcal{H}]^{\beta} := \left\{ \sum_{j=1}^{\infty} f_j \mu^{\beta/2}_j \phi_j: \sum_{j=1}^{\infty} |f_j|^2 < \infty \right\}.
\end{align*}
\end{enumerate}
The space $[\mathcal{H}]^{\beta}$ corresponds to the range of the linear operator $f \mapsto \sum_{j=1}^{\infty} \mu_j^{\beta/2} f_j \phi_j$ for $f = \sum_j f_j \phi_j$ with $\sum_i |f_i|^2<\infty$, i.e., $f \in \mathcal{L}_v^2(X)$. Equipped with the $\beta$-power inner product
\begin{align*}
\bigg \langle\sum_{j=1}^{\infty} f_j \mu^{\beta/2}_j \phi_j,  \sum_{j=1}^{\infty} g_j \mu^{\beta/2}_j \phi_j \bigg \rangle_{\beta} := \sum_{j=1}^{\infty} f_j g_j,
\end{align*}
the space $[\mathcal{H}]^{\beta}$ is isometrically isomorphic to the sequence space $\ell^2(\mathbbm{N})$ and as such a separable Hilbert space on its own right. The limiting cases $\beta = 0$ and $\beta=1$ are of great interest, as they correspond, on the one hand, to the space $\mathcal{L}^2_v(X)$ and, on the other hand, to $\mathcal{H}$ (recall the equivalent definition \eqref{eq:rkhs}). Since $\mu_j \to 0$, increasing $\beta$ dampens high frequencies $f_j \phi_j$ bringing about smoother functions. Thus, by allowing $\beta \in (0,1]$ we can capture different degrees of smoothness and misspecification on a continuous scale.

A technical condition required in our analysis is the existence of an $\alpha \in (0, \beta)$ such that  $[\mathcal{H}]^{\alpha}$ is continuously embedded into $\mathcal{L}_{v}^{\infty}(X)$, the space of $v$-essentially bounded functions on $X$.
\begin{enumerate} [label=(A\arabic*), ref=(A\arabic*)]
\setcounter{enumi}{5}
\item \label{A6} There exists an $\alpha \in (0,\beta)$ such that the inclusion  $[\mathcal{H}]^{\alpha} \hookrightarrow \mathcal{L}^{\infty}_v(X)$ is continuous.
\end{enumerate}
Assumption \ref{A6} is the critical assumption (EMB) in \citep{Fischer:2020} with $\alpha < \beta$; it is not needed when $\beta = 1$ in which case our analysis only requires \ref{A1}--\ref{A4}. With these definitions, our main result for this section is as follows. 

\begin{theorem}
\label{thm:1}
Suppose that assumptions \ref{A1}--\ref{A6} are satisfied and $\lambda \to 0$ in such a way that $\mathcal{N}^2_{\infty}(\lambda)/n \to 0$ and $ \lambda^{\beta}\mathcal{N}_{\infty}(\lambda)	 \to 0$, as $n \to \infty$. Then,
\begin{align*}
\int_{X} \{\widehat{f}_n(x) - f_0(x)\}^2 v(dx) = O_{\mathbb{P}}\left( \frac{\mathcal{N}_{\infty}(\lambda)}{n} \right) + O_{\mathbb{P}} ( \lambda^{\beta}),
\end{align*}
as $n \to \infty$.
\end{theorem}

Theorem~\ref{thm:1} establishes an asymptotic decomposition of the $\mathcal{L}^2_v(X)$-error into two terms, which correspond to the asymptotic variance and the squared penalization bias respectively. It is interesting to observe that the variance $\mathcal{N}_{\infty}(\lambda)/n$ is unaffected by misspecification, i.e., by $\beta < 1$. Instead, misspecification only affects the penalization bias $\lambda^{\beta}$. Given that $\lambda \to 0$, the latter is larger for smaller $\beta \in (0,1]$. To illustrate the result with concrete rates of convergence, consider first the Sobolev space $\mathcal{H}^{m}(X)$ discussed in Example K4. The variance term is $(n \lambda^{d/2m})^{-1}$ and is balanced with the squared bias $\lambda^{\beta}$ when $\lambda \asymp n^{-2m/(2m\beta+d)}$ leading to the rate $n^{-2m\beta/(2m\beta+d)}$. By \citep[Theorem 2]{Fischer:2020} this is the optimal (minimax) rate of convergence. Consider next the zonal kernel from Example K2. The best rate of convergence is $[\log(n)]^{d-1}/n$ obtained when $\lambda \asymp n^{-1/\beta} [\log(n)]^{(d-2)/\beta}$. Owing to the smoothness of the kernel, the rate is almost parametric and is unaffected by misspecification. The minimax rates for exponentially decaying eigenvalues are not known, but we conjecture that, up to a logarithmic power, this rate is optimal.

It is worth emphasizing that, while the learning literature has overwhelmingly focused on the least-squares estimator with $\rho(x) = x^2/2$, Theorem~\ref{thm:1} holds for all loss functions satisfying \ref{A2}--\ref{A4}. As noted, the main technical hurdle in the analysis of these estimators is the lack of closed-form solution upon which most theoretical results to date are based, see, e.g., \citep{Smale:2005, Zhang:2023}. The dependence on this closed form solution  means that none of those arguments carries over to the general situation. We overcome this difficulty by developing an asymptotic linearisation of the objective function, which is delicate, as we do not assume that $\rho$ is smooth. We study the objective function in a shrinking neighbourhood of the abstract regularized interpolant (defined in the appendix) and show that there exists a minimizer in the interior, which by strict convexity can only be $\widehat{f}_n$. The interested reader is referred to the appendix for the complete argument.

The present treatment is novel even compared to the modern analysis of penalized M-estimators with empirical process theory  \citep[see, e.g.,][Chapter 12]{VDG:2000} in two important ways. First, we do not require that $\rho$ should be Lipschitz. The Lipschitz condition is widely used to estimate the Rademacher complexity by contraction, an elegant step simplifying many proofs, but excluding widely used losses such as expectile and $L^p$ losses with $p>1$. We take a different approach by using the more general \ref{A2} and \ref{A3} to decompose the objective function. We apply the contraction principle on the remainder of the linearisation rather than on the standard symmetrised empirical process. Secondly, our approach leads to an informative error decomposition as opposed to a single rate $n^{-s}$ one typically obtains through empirical process theory alone. We believe that such an error decomposition is valuable, as, on the one hand, it makes the bias-variance trade-off explicit, and on the other hand, it is the first step towards second order inference, namely the construction of confidence bands for $f_0$. 

\section{Tensor product RKHS and dominating mixed smoothness}
\label{sec:tensor}

Having established the general asymptotic properties of regularized M-estimators on an RKHS $\mathcal{H}$ with domain $X$, we now turn to the intriguing tensor construction allowing one to model data on Cartesian domains $X \times X \ldots \times X$. While the general recipe to be described applies to any RKHS, even with multidimensional domain $X$, we focus here on the Sobolev space $\mathcal{H}^m(X)$ with $X \subset \mathbbm{R}$ due to its connection with flexible modelling. A particularly appealing feature of tensor Sobolev spaces is their ability to naturally incorporate interaction effects between variables; the interested reader is referred to \citep{Lin:2000} for insightful discussions and results for $\rho(x) = x^2/2$. Here, we briefly describe the tensor construction and illustrate how our methodology applies and yields new results recovering and vastly extending those of \citep{Lin:2000}.

Our precise setting is as follows. Let $X \subset \mathbbm{R}$ be compact and connected, say $X=[0,1]$ without loss of generality. The standard Sobolev space $\mathcal{H}^m([0,1])$ (with its usual inner product and norm) is an RKHS for all $m \in \mathbbm{N}$ with kernel $K:[0,1] \times [0,1] \to \mathbbm{R}$. As discussed in Example K1, the eigenfunctions of $K$, $\{\phi_j\}_{j=1}^{\infty}$, form a complete orthonormal system for $\mathcal{L}^2([0,1])$ and $\{\mu_j^{1/2} \phi_j\}_{j=1}^{\infty}$ is a complete orthonormal system for $\mathcal{H}^m([0,1])$. The $d$-dimensional Sobolev tensor product space 
$\mathcal{H}^m([0,1])^{\otimes d}$ is defined as
\begin{align*}
\mathcal{H}^m([0,1])^{\otimes d} := \left\{f: [0,1]^d \to \mathbbm{R}: f = \sum_{i_1, \ldots i_d} f_{i_1, \ldots, i_d} \phi_{i_1}(x_1) \ldots \phi_{i_d}(x_d), \ \sum_{i_1, \ldots, i_d} \frac{|f_{i_1, \ldots, i_d}|^2}{\mu_{i_1} \ldots \mu_{i_d}} < \infty \right\},
\end{align*}
where each of the $d$ sums runs over $\mathbbm{N}$. This space results from the closure of the linear span  of the products $\{\phi_{i_1} \ldots \phi_{i_d}\}_{i_1, \ldots, i_d}$ under the inner product  $\langle \phi_{i_1} \ldots \phi_{i_d}, \phi_{j_1} \ldots \phi_{j_d} \rangle_{\mathcal{H}^m(X)^{\otimes d}} := \prod_{k=1}^d \frac{\delta_{i_k j_k}}{\mu_{i_k}}$ extended to finite sums. It can be shown that $\mathcal{H}^m([0,1])^{\otimes d}$ is isometrically isomorphic to $\ell^2(\mathbbm{N}^d)$ and as such separable and complete. It is also an RKHS with kernel 
\begin{align*}
\widetilde{K}(x,y) := \prod_{i=1}^d K(x_i,y_i), \quad x,y \in [0,1]^d,
\end{align*}
\citep[see, e.g.,][Lemma 4.6]{St:2008}. Moreover, when $K$ is positive definite on $[0,1]$, then, by a well-known theorem of I. Schur, so is $\widetilde{K}$ on $[0,1]^d$.

As previously, it is convenient to equip $\mathcal{H}^m([0,1])^{\otimes d}$ with a penalty-weighted inner product: $ \langle f, g \rangle_{\mathcal{H}^m([0,1])^{\otimes d}, \lambda} := \langle f, g \rangle_{\mathcal{L}^2([0,1]^d)} + \lambda \langle f, g \rangle_{\mathcal{H}^m([0,1])^{\otimes d}}$. The reproducing kernel for this inner product is
\begin{align}
\label{eq:tRK}
\widetilde{K}_{\lambda}(x,y) = \sum_{i_1,\ldots, i_d} \frac{\phi_{i_1}(x_1) \ldots \phi_{i_d}(x_d)\phi_{i_1}(y_1) \ldots \phi_{i_d}(y_d)}{1+\lambda \prod_{j=1}^d \gamma_{i_j}}, \quad x, y \in [0,1]^d,
\end{align}
where $\gamma_{i_j} = \mu_{i_j}^{-1}$ and the series converges absolutely and uniformly for each $\lambda>0$.  Our main result for this section is presented in Theorem~\ref{thm:2} below.

\begin{theorem}
\label{thm:2}
Suppose that $\mathcal{H} = \mathcal{H}^m([0,1])^{\otimes d}$ and the predictors $\{x_i\}_{i=1}^n$ are i.i.d. vectors on $[0,1]^d$ with distribution $v$ having Lebesgue density bounded away from zero and infinity. Suppose also that  assumptions \ref{A2}--\ref{A6} are satisfied and $\lambda \to 0$ in such a way that $\lambda^{-1/m} [\log(1/\lambda)]^{2(d-1)}/n \to 0$ and $\lambda^{\beta-1/(2m)} [\log(1/\lambda)]^{d-1} \to 0$. Then,
\begin{align*}
\int_{[0,1]^d} \{\widehat{f}_n(x) - f_0(x)\}^2 dx = O_{\mathbb{P}}\left( \frac{\left[\log\left(\frac{1}{\lambda} \right)\right]^{d-1} }{n \lambda^{\frac{1}{2m}}} \right) + O_{\mathbb{P}} ( \lambda^{\beta}),
\end{align*}
as $n \to \infty$.
\end{theorem}

To appreciate how striking this result is, recall that by Example K4 and Theorem~\ref{thm:1}, had our target been in $\mathcal{H}^m([0,1]^d)$ (the standard isotropic Sobolev space on $[0,1]^d$), the variance term of the RKHS estimator would have been $(n\lambda^{d/(2m)})^{-1}$, and this is, in general, unimprovable. By contrast, for an $\mathcal{H}^m([0,1])^{\otimes d}$ target, the variance of the tensor RKHS estimator is, up to a log term, $(n\lambda^{1/(2m)})^{-1}$. The latter is substantially smaller as $\lambda \to 0$, even though the domains are the same. In this light, Theorem~\ref{thm:2} appears to suggest that nonparametric regression in tensor RKHS avoids the curse of dimensionality. However, the reality is more subtle, because $\mathcal{H}^m([0,1])^{\otimes d}$ consists of substantially smoother functions than those in $\mathcal{H}([0,1]^d)$.

The crucial observation explaining this puzzling phenomenon is that the tensor product space $\mathcal{H}^m([0,1])^{\otimes d}$ can be identified, up to equivalent norms, with the Sobolev space with dominating mixed smoothness on $[0,1]^d$ of order $m$, denoted with $S^m([0,1]^d)$, well-known in approximation theory \citep{Triebel:2019}. As, to the best of our knowledge, this formal identification on bounded product domains has not been established previously, we state it in Proposition~\ref{prop:3} and present the proof.

\begin{prop}
\label{prop:3}
The tensor product Sobolev space $\mathcal{H}^m([0,1])^{\otimes d}$ may be identified with the Sobolev space with dominating mixed smoothness $S^m([0,1]^d)$ with equivalent norms.
\end{prop}

\begin{proof}

Our proof traces the argument in the proof of Corollary 10.48 of \citep{Wend:2005} for isotropic Sobolev spaces. Call $S^m(\Omega^d)$ the Sobolev space of order $m$ of dominating mixed smoothness on $\Omega^d$ with $\Omega \subset \mathbbm{R}$. Every  $ f\in \mathcal{H}^m([0,1])^{\otimes d}$ has a norm preserving extension $Ef \in \mathcal{H}^m(\mathbbm{R})^{\otimes d}$ \citep[Theorem 10.46]{Wend:2005}, because $\mathcal{H}^m([0,1])^{\otimes d}$ is an RKHS with the product kernel $\widetilde{K}(x,y) = \prod_{i=1}^d K(x_i,y_i)$. We also know, by Theorem 2.1 in \citep{Sickel:2009}, that $\mathcal{H}^m(\mathbbm{R})^{\otimes d} = S^m(\mathbbm{R}^d)$ with equivalent norms. Thus, we have  $f = Ef|[0,1]^d \in S^m([0,1]^d)$ and
\begin{align*}
\|f\|_{S^m([0,1]^d)} \leq \|Ef\|_{S^m(\mathbbm{R}^d)} \leq c  \|Ef\|_{\mathcal{H}^m(\mathbbm{R})^{\otimes d}} = c \|f\|_{\mathcal{H}^m([0,1])^{\otimes d}},
\end{align*}
because the extension $E$ from $[0,1]^d$ to $\mathbbm{R}^d$ is norm preserving. On the other hand, on $S^m([0,1]^d)$,  by \citep[Proposition 1.25]{Triebel:2019}, there exists a bounded linear extension operator $\widetilde{E}:S^m([0,1]^d) \to S^m(\mathbbm{R}^d)$. This means that, for every $f \in S^m([0,1]^d)$,  $\widetilde{E}f \in S^m(\mathbbm{R}^d) = \mathcal{H}^m(\mathbbm{R})^{\otimes d}$, with equivalent norms by \citep[Theorem 2.1]{Sickel:2009}. Moreover, by the boundedness of $\widetilde{E}$, $\|\widetilde{E}f\|_{S^m(\mathbbm{R}^d)} \leq C \|f\|_{S^m([0,1]^d)}$. Thus, $f = \widetilde{E}f|[0,1]^d \in \mathcal{H}^m([0,1])^{\otimes d}$, by the restriction theorem \citep[Theorem 10.47]{Wend:2005}. This now leads to
\begin{align*}
\|f\|_{\mathcal{H}^m([0,1])^{\otimes d}} \leq \|\widetilde{E}f\|_{\mathcal{H}^m(\mathbbm{R})^{\otimes d}} \leq c\|\widetilde{E}f\|_{S^m(\mathbbm{R}^d)}  \leq C \|f\|_{S^m([0,1]^d)},
\end{align*}
where the first inequality follows from \citep[Theorem 10.47]{Wend:2005}. The two chains of inequalities show that the norms $\|\cdot\|_{\mathcal{H}^m([0,1])^{\otimes d}}$ and $\|\cdot\|_{S^m([0,1]^d)}$ are equivalent, which completes the proof.
\end{proof}

It is interesting to note that, using the identification in Proposition \ref{prop:3} and the classical embedding $\mathcal{H}^{dm}([0,1]^d) \subset S^m([0,1]^d)$, we have the continuous inclusion
\begin{align*}
\mathcal{H}^{dm}([0,1]^d) \subset \mathcal{H}^m([0,1])^{\otimes d},
\end{align*}
which is the best possible. Therefore, we may heuristically view $\mathcal{H}^m([0,1])^{\otimes d}$ as ``almost" equal to $\mathcal{H}^{dm}([0,1]^d)$ and should expect an asymptotic variance of order $(n \lambda^{d/(2dm)})^{-1} = (n \lambda^{1/(2m)})^{-1}$. Up to a log-term, this is precisely the variance obtained in Theorem~\ref{thm:2}. This observation explains why tensor product Sobolev spaces appear to sidestep the curse of dimensionality: because their smoothness increases with the dimension.

Specializing Theorem~\ref{thm:2} to the case $f_0 \in \mathcal{H}^m([0,1])^{\otimes d}$, that is, $\beta = 1$, and taking $\lambda \asymp \{n[\log(n)]^{(1-d)}\}^{-2m/(2m+1)}$, we obtain
\begin{align*}
\int_{[0,1]^d} \{\widehat{f}_n(x) - f_0(x)\}^2 dx = O_{\mathbb{P}} \left( \left\{n [\log(n) ]^{1-d} \right\}^{-\frac{2m}{2m+1}}\right),
\end{align*}
which is in agreement with the rate in \citep[Theorem 1.2]{Lin:2000} derived for the special case $\rho(x) = x^2/2$, also assuming that $v$ possesses Lebesgue density bounded away from zero and infinity. That author also proved that this is the minimax (optimal) rate of convergence \citep[Theorem 1.1]{Lin:2000}. Hence, greatly extending existing results, Theorem~\ref{thm:2} establishes the minimax optimality of an entire family of tensor RKHS estimators under mild assumptions.

\section{Practical implementation}

The estimators can be readily implemented using the representer theorem, which guarantees the finite-dimensional representation $\widehat{f}_n(\cdot) = \sum_{i=1}^n K(x_i, \cdot) \beta_i$, cf. Proposition~\ref{prop:1}. Its squared $\mathcal{H}$-norm $\|\widehat{f}_n\|_{\mathcal{H}}^2$ can be shown from first principles to equal $\beta^{\top} K \beta$ where $K \in \mathbbm{R}^{n \times n}$ with $K_{ij} = K(x_i,x_j)$. Plugging this representation into the objective function \eqref{eq:est}, we need only determine
\begin{align*}
\widehat{\beta}_n = \argmin_{\beta \in \mathbbm{R}^n} \left[ \frac{1}{n} \sum_{i=1}^n \rho(y_i- K_i^{\top} \beta) + \lambda \beta^{\top} K \beta \right],
\end{align*}
thereby transforming \eqref{eq:est} to a finite-dimensional problem. For each given $\lambda>0$, the coefficient vector $\widehat{\beta}_n$ may be identified through the well-established iteratively reweighted least-squares (IRLS) algorithm, see, e.g., \citep[Chapter 9]{Mar:2019}. For convex losses, the algorithm is guaranteed to converge to $\widehat{\beta}_n$, but requires that $\lim_{x \to 0} \rho^{\prime}(x)/x$ exists, which is not the case for the quantile loss $\rho_{\tau}(x) = x(\tau-1(x<0))$, $\tau \in (0,1)$. For this reason, we replace $\rho_{\tau}$ with its local quadratic approximation about $0$ proposed by \citet{Nychka:1995}, viz,
\begin{align*}
\widetilde{\rho}_{\tau,\epsilon}(x) = \begin{cases} \rho_{\tau}(x), & |x|>\epsilon, \\ \tau x^2/(2\epsilon), & 0 \leq x \leq \epsilon, \\ (1-\tau)x^2/(2\epsilon), & -\epsilon \leq x \leq 0. \end{cases}
\end{align*}
Here, $\epsilon>0$ is a (small) tuning parameter. It is clear that $\|\widetilde{\rho}_{\tau,\epsilon}-\rho_{\tau}\|_{\infty} = \epsilon \max(\tau,(1-\tau))/2$ and so smaller values of $\epsilon$ provide better approximations to $\rho_{\tau}$; we take $\epsilon = 10^{-3}$ in our implementations. The IRLS is applicable for $\widetilde{\rho}_{\tau,\epsilon}$ and, in our experience, converges fast to the minimizer $\widehat{\beta}_n$.
\begin{figure}[H]
\centering
\subfloat{\includegraphics[width=0.48\textwidth]{"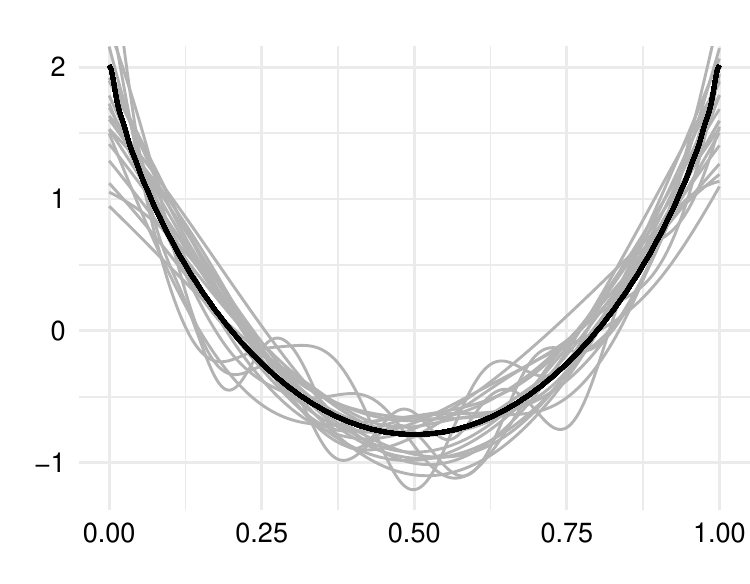"}}
\subfloat{\includegraphics[width=0.48\textwidth]{"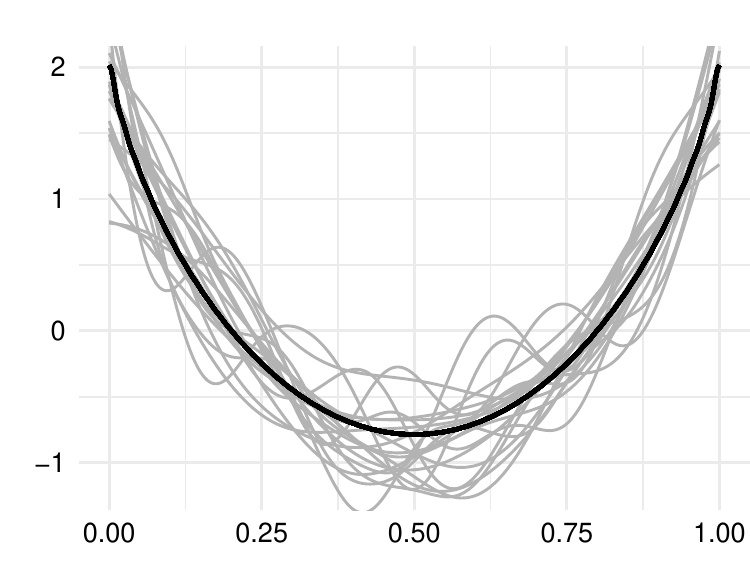"}}
\caption{20 estimates with $\beta = 0.5$ and Gaussian errors of the least squares (left) and least absolute deviations (right) estimators. The solid black line represents the true regression function.}
\label{fig:1}
\end{figure}
To select the penalty parameter $\lambda>0$, we propose to minimize the robust leave-one-out cross-validation criterion given by
\begin{align*}
\RCV(\lambda) = \sum_{i=1}^n  \frac{w_i(\lambda) r_i^2(\lambda)}{(1-h_i(\lambda))^2},
\end{align*}
where $r_i(\lambda)$ are the residuals: $r_i(\lambda) = y_i - K_i^{\top} \widehat{\beta}_n(\lambda)$, $w_i(\lambda)$ are weights defined as $w_i(\lambda) = \rho^{\prime}(r_i(\lambda))/r_i(\lambda)$ and $h_i(\lambda)$ are the diagonal elements of the weighted hat matrix obtained upon convergence of the IRLS algorithm, that is, $H(\lambda) = K(K W(\lambda) K + 2n\lambda K)^{-1} K W(\lambda)$ where $W(\lambda) = \diag(w_1(\lambda), \ldots, w_n(\lambda))$. Note that $K$ is positive definite, so we may rewrite this as $H(\lambda) = K(W(\lambda) K + n\lambda \mathcal{I}_{n \times n})^{-1}W(\lambda)$. If $w_i(\lambda)=1$ for all $i=1, \ldots, n$, $\RCV(\lambda)$ reduces to the commonly used leave-one-out cross-validation. Fast implementations of the least-squares, least absolute deviations and Huber estimators as well as  \textsf{R}-scripts reproducing the tables and figures of this paper may be found at  \url{https://github.com/ioanniskalogridis/Generalized-nonparametric-regression-in-RKHS}. 

\begin{figure}[H]
\centering
\subfloat{\includegraphics[width=0.48\textwidth]{"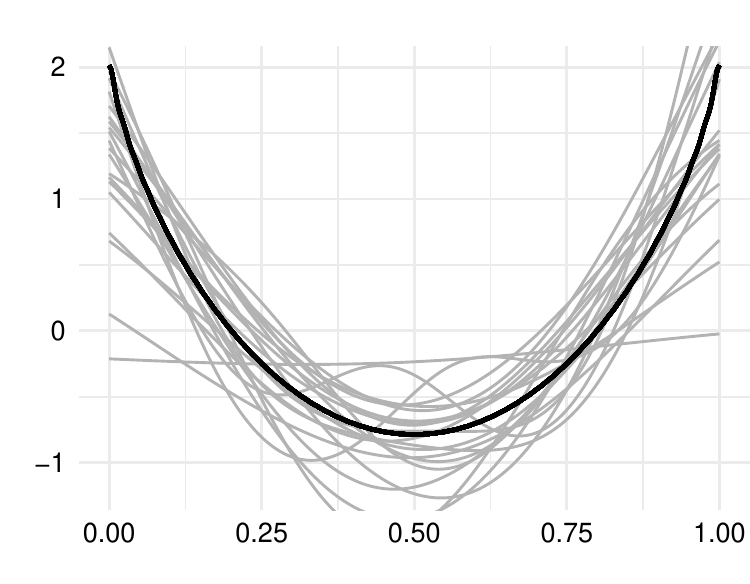"}}
\subfloat{\includegraphics[width=0.48\textwidth]{"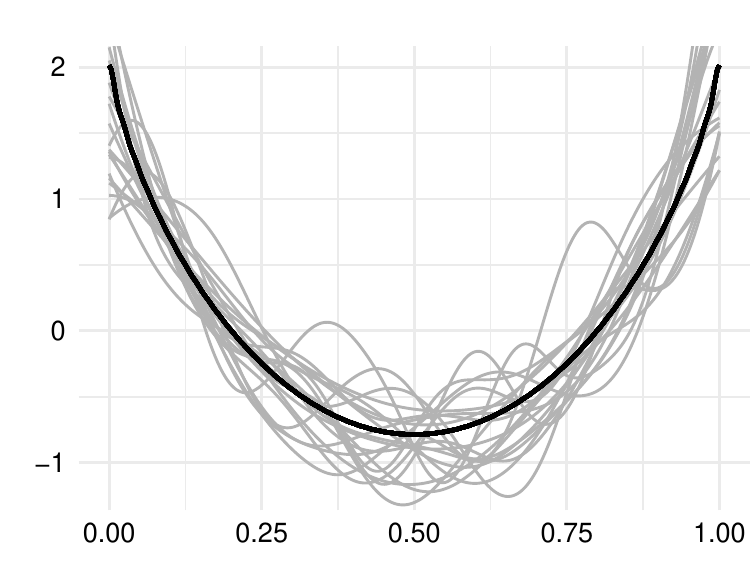"}}
\caption{20 representative estimates with $\beta = 0.5$ and $t_2$ errors of the least squares (left) and least absolute deviations (right) estimators. The solid black line represents the true regression function.}
\label{fig:2}
\end{figure}

\section{Numerical illustrations}

In our numerical experiments, we study the effects of (i) heavy-tailed errors $\{\epsilon_i\}_{i=1}^n$ (ii) target function roughness and (iii) dimensionality on the estimates. The estimators to be compared are
\begin{itemize}
\item The least-squares estimator with $\rho(x) = x^2/2$ in \eqref{eq:est}, abbreviated as LS.
\item The least absolute deviations estimator with $\rho(x) = |x|$ in \eqref{eq:est}, abbreviated as LAD.
\item The Huber estimator with $\rho_{k}(x) = (x^2/2) 1(|x| \leq k) + k(|x| - k/2) 1(|x|>k)$ with $k=1.345$ in \eqref{eq:est}, abbreviated as Huber.
\end{itemize}
To ensure that the estimators are comparable we have implemented them ourselves with an \textsf{R}-interface \citep{R:2026}, but with a C++ back-end for enhanced computational efficiency. We begin with the one-dimensional setting. Here, the observations have been generated according to
\begin{align*}
y_i = \sum_{j=1}^{50} j^{-2 \beta} j^{-2/3} \sqrt{2}\cos(2 \pi(j-1) x_i) + \epsilon_i, \quad (i=1,\ldots, n),
\end{align*}
with $\{x_i\}_{i=1}^n \overset{iid}{\sim} \mathrm{Unif}[0,1]$ and Gaussian or student $t_2$ errors $\{\epsilon_i\}_{i=1}^n$. The parameter $\beta>0$ controls the decay of Fourier coefficients and hence the smoothness of the regression function. We estimate the regression function $f_0(x):=\sum_{j=1}^{50} j^{-2 \beta} j^{-2/3} \cos(2 \pi (j-1) x)$ using the Mat\'ern kernel with $v=1.5$ corresponding, up to equivalent norms, to the Sobolev space $\mathcal{H}^2([0,1])$, see Section~\ref{sec:kernel}. A useful spectral heuristic is $\mu_j \asymp j^{-4}$ and so we consider values of $\beta \in \{0.5, 0.75, 1\}$ reflecting severe misspecification ($\beta = 0.5$), mild misspecification ($\beta = 0.75)$ and correct specification $(\beta = 1)$, respectively. We evaluate performance with the mean-squared error $\MSE = 300^{-1} \sum_{j=1}^{300} (\widehat{f}_n(t_i)-f_0(t_i))^2$ where the $\{t_i\}_{i=1}^{300}$ are equidistant within $[0,1]$.

\begin{table}[H]
\centering
\caption{Mean-squared errors and associated standard errors ($\times 10$) of the estimators over 500 replications.}
\begin{tabular}{ccccccccc}
\toprule
\multicolumn{3}{c}{ } & \multicolumn{2}{c}{LS} & \multicolumn{2}{c}{LAD} & \multicolumn{2}{c}{Huber} \\
\cmidrule(l{3pt}r{3pt}){4-5} \cmidrule(l{3pt}r{3pt}){6-7} \cmidrule(l{3pt}r{3pt}){8-9}
$d$ & $\beta$ & Noise & MSE & SE & MSE & SE & MSE & SE\\
\midrule
\multirow{6}{*}{$d = 1$} & \multirow{2}{*}{$\beta = 0.50$} & Gaussian & 0.3617 & 0.0101 & 0.5945 & 0.0145 & 0.3850 & 0.0099\\
 &  & $t_2$ & 2.3087 & 0.1902 & 0.8076 & 0.0216 & 0.5509 & 0.0137\\
 & \multirow{2}{*}{$\beta = 0.75$} & Gaussian & 0.2979 & 0.0121 & 0.6070 & 0.0160 & 0.3115 & 0.0099\\
 &  & $t_2$ & 3.3263 & 1.1597 & 0.7836 & 0.0207 & 0.4628 & 0.0154\\
 & \multirow{2}{*}{$\beta = 1.00$} & Gaussian & 0.2671 & 0.0096 & 0.5807 & 0.0149 & 0.2907 & 0.0104\\
 &  & $t_2$ & 1.8748 & 0.0976 & 0.7585 & 0.0233 & 0.4338 & 0.0145\\
\midrule
\multirow{6}{*}{$d = 2$} & \multirow{2}{*}{$\beta = 0.50$} & Gaussian & 1.2275 & 0.0173 & 1.4504 & 0.0165 & 1.2409 & 0.0140\\
 &  & $t_2$ & 3.1226 & 0.1074 & 1.6950 & 0.0246 & 1.4827 & 0.0202\\
 & \multirow{2}{*}{$\beta = 0.75$} & Gaussian & 1.2153 & 0.0180 & 1.4523 & 0.0174 & 1.2738 & 0.0135\\
 &  & $t_2$ & 3.2979 & 0.1204 & 1.7609 & 0.0241 & 1.6425 & 0.0215\\
 & \multirow{2}{*}{$\beta = 1.00$} & Gaussian & 1.2353 & 0.0185 & 1.5179 & 0.0178 & 1.3370 & 0.0163\\
 &  & $t_2$ & 3.2482 & 0.1211 & 1.8422 & 0.0266 & 1.7667 & 0.0213\\
\bottomrule
\end{tabular}
\label{tab:1}
\end{table}

We construct an analogous two-dimensional model using tensor-product Fourier bases. In particular, we generate data according to
\begin{align*}
y_i = 2\sum_{j=1}^{50} \sum_{k=1}^{50} (jk)^{-2\beta} (jk)^{-2/3} \sin(2 \pi j x_{i1} ) \sin(2 \pi k x_{i2}) + \epsilon_i, \quad(i=1, \ldots, n)
\end{align*}
with $\{x_i\}_{i=1}^n \overset{iid}{\sim} \mathrm{Unif}[0,1]^2$ and  Gaussian or $t_2$ errors $\{\epsilon_i\}_{i=1}^n$. The regression function is estimated using a tensor-product Mat\'ern kernel on $[0,1]^2$ with $v=1.5$, as described in Section~\ref{sec:tensor}. We again consider $\beta \in \{0.5,0.75,1\}$ and evaluate performance with the mean-squared error $\MSE = 50^{-2} \sum_{j=1}^{50} \sum_{k=1}^{50} (\widehat{f}_n(t_j,t_k)-f_0(t_j,t_k))^2$ where the $\{t_{j}\}_{j=1}^{50}$ are equidistant within $[0,1]$. Table~\ref{tab:1} reports the results from 500 replications.

Table~\ref{tab:1} lends itself to several interesting findings the most notable of which is the deterioration in performance of the LS estimator with heavy-tailed errors. In the case of light-tailed Gaussian errors, the LS estimator outperforms the robust estimators LAD and Huber, though the latter only by a small margin.  However, under heavy-tailed $t_2$ errors, the performance of the LS estimator deteriorates nearly 10-fold and the robust estimators convincingly outperform it.  The vulnerability of the LS estimator to heavy-tailed observations can also be attested visually, as in Figures~\ref{fig:1}--\ref{fig:4}. Comparing the robust estimators in detail, it may be seen that the Huber estimator is more efficient than the LAD estimator under light-tailed Gaussian errors, the inefficiency of the latter under such errors being well-known. Under $t_2$ errors, LAD and Huber perform roughly similarly with the latter having a slight edge due to the symmetry of the $t_2$-distribution.
 
\begin{figure}[H]
\centering
\includegraphics[width=0.95\textwidth]{"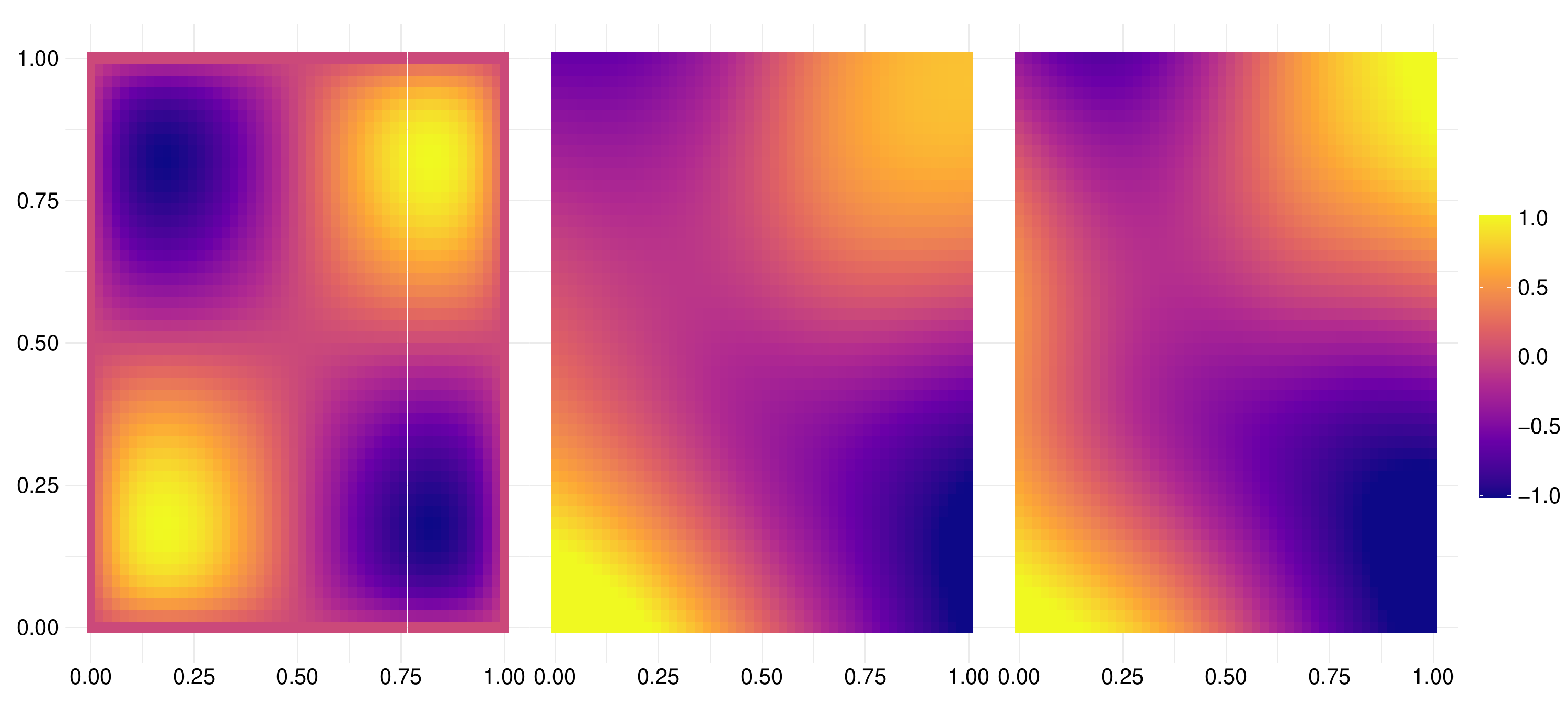"}
\caption{Contours of the two-dimensional true regression function (left) and representative least squares (middle) and least absolute deviation estimates (right) for $\beta = 0.5$ and Gaussian errors.}
\label{fig:3}
\end{figure}

As predicted by the theory, a larger $\beta$ improves performance for all estimators, but its effect is less notable for $d=2$. An explanation for this phenomenon is that, for multidimensional data and moderate sample sizes, it is the variance that is the main driver of the mean-squared error, itself an approximation of the $\mathcal{L}^2$-error, rather than the bias, which is the quantity affected by $\beta$.

\section{Discussion}

The results of this paper provide theoretical and practical justification for the use of a broad class of regularized estimators that include the least-squares estimator as a special case. Our analysis extends and unifies existing literature through the novel concept of spectral complexity while treating loss functions that have received little attention in the literature. 

Several promising directions remain open. Most immediately, the methodology developed in this work can be extended to conditionally positive definite kernels, such as Duchon's thin-plate splines \citep{Duchon:1977}, which are widely used in scattered data approximation and spatial statistics. Another natural avenue is the application to functional and longitudinal data analysis, where one can construct novel regularized estimators for location and dispersion from discretely sampled functional data, building upon and extending existing work \citep[e.g.,][]{Kal:2023}.

\begin{figure}[H]
\centering
\includegraphics[width=0.95\textwidth]{"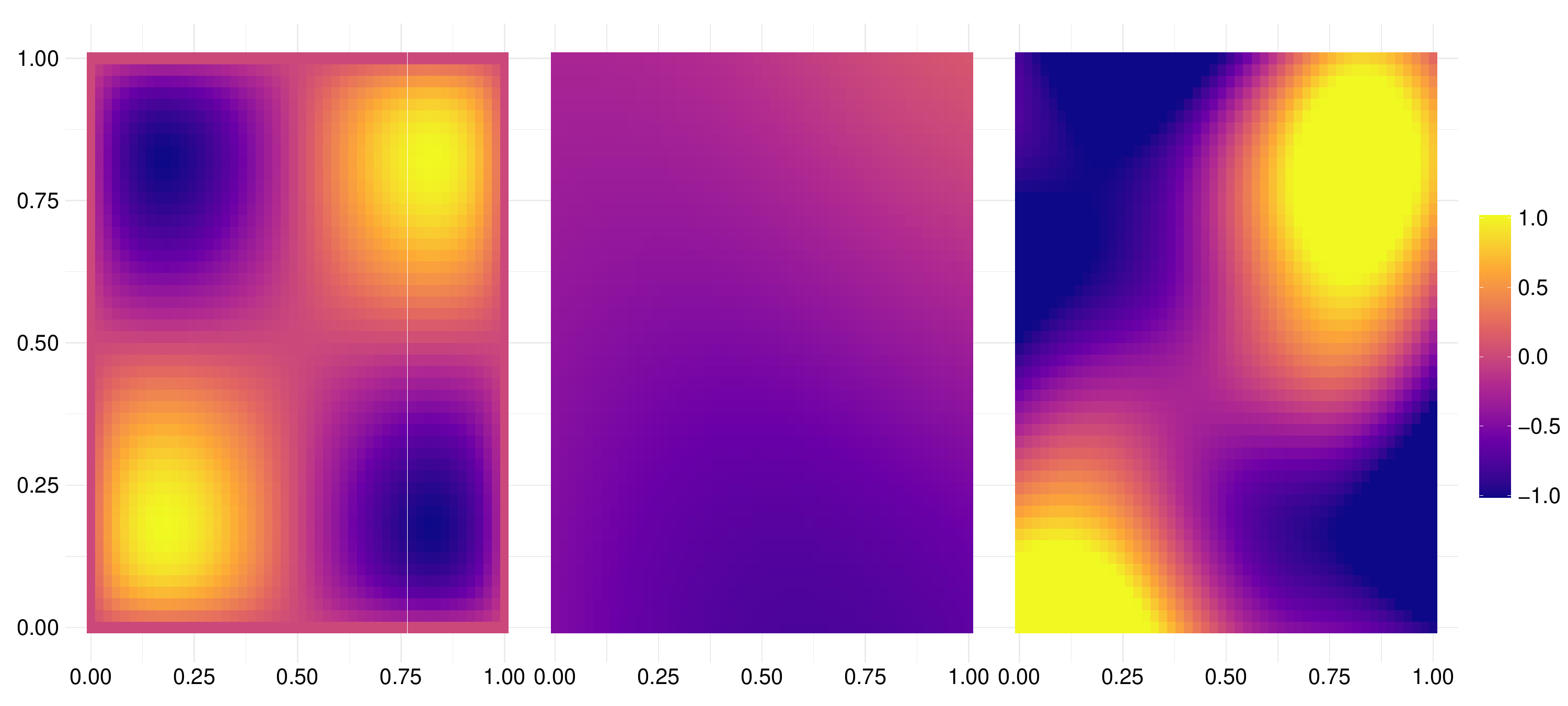"}
\caption{Contours of the two-dimensional true regression function (left) and representative least squares (middle) and least absolute deviation estimates (right) for $\beta = 0.5$ and $t_2$ errors.}
\label{fig:4}
\end{figure}

\section{Appendix: Detailed proofs}

To lighten the notation in the course of our proofs we use $c_0$ to denote generic positive constants whose precise value is unimportant. Thus, the value of $c_0$ may change from appearance to appearance. We will occasionally use inequalities of the form $A \leq B$ for $A, B$ random variables defined on the same probability space $(\Omega,\mathcal{A},\mathbb{P})$. These inequalities are to be interpreted pointwise, i.e., $A(\omega) \leq B(\omega)$ for all $\omega \in \Omega$.

The proof of Theorem~\ref{thm:1} uses the abstract regularized interpolant of $f_0$, which we now recall. Let $S_K: \mathcal{L}^2_v(X) \to \mathcal{H}$ denote the integral operator 
\begin{align*}
S_K(f) (s) = \int_{X} K(s,t) f(t) v(dt), \quad f \in \mathcal{L}^2_v(X).
\end{align*}
Note that when the range of the integral operator is taken to be in $\mathcal{L}^2_v(X)$, this operator is typically denoted with $T_K$ (as we have also done in the main text). Let $S_K^{\star}: \mathcal{H} \to \mathcal{L}^2_v(X)$ denote its Hilbert adjoint, which is the inclusion operator $\mathcal{I}$, see Lemma 2.2 of \citep{St:2012} or Proposition 10.28 of \citep{Wend:2005}. Assumption~\ref{A1} ensures that the inclusion operator $\mathcal{I}$ is injective so $\ker(\mathcal{I}) = \{0\}$, which implies $\overline{\Ran(S_K)} = \ker(\mathcal{I})^{\perp} = \mathcal{H}$.

Define the operator $C_K := S_K S_K^{\star}$ mapping $\mathcal{H}$ into itself. The abstract regularized interpolant is defined as
\begin{align*}
f_{\lambda} = (C_K+\lambda)^{-1} S_K(f_0),
\end{align*}
and satisfies $f_{\lambda} \in \mathcal{H}$ even when $f_0 \notin \mathcal{H}$.
This interpolant has been studied extensively, e.g., by \citet{Smale:2005} and \citet{Fischer:2020}. The following lemma collects three properties relevant to our development; for a proof of these statements the reader is referred to Lemma 14 and Corollary 15 of \citep{Fischer:2020}.

\begin{lemma}
\label{lem:intprop}
Under assumptions \ref{A5} and \ref{A6}, for all $\lambda \in (0,1]$ the abstract regularized interpolant $f_{\lambda}$ satisfies the following properties:
\begin{enumerate}
\item  Controlled explosion in $\mathcal{H}$: $\|f_{\lambda}\|_{\mathcal{H}}^2 \leq c_0 \lambda^{-(1-\beta)}$.
\item Uniform convergence in $X$: $\|f_{\lambda}-f_0\|_{\infty}^2 \leq c_0 \lambda^{\beta-\alpha}$. 
\item Convergence in $\mathcal{L}^2_v(X)$: $\|f_{\lambda}-f_0\|_{\mathcal{L}^2_{v}(X)}^2 \leq c_0 \lambda^{\beta}$.
\end{enumerate}
\end{lemma}

We now present the proof of Theorem~\ref{thm:1}.

\begin{proof}[Proof of Theorem~\ref{thm:1}] 
Write $R_i =  f_0(x_i) - f_{\lambda}(x_i)$ for $i=1, \ldots, n$ and let $L_n$ denote the objective function, that is,
\begin{align*}
L_n(f)  = \frac{1}{n} \sum_{i=1}^n \rho\left(y_{i}-f(x_i) \right) + \lambda \|f\|_{\mathcal{H}}^2, \quad f \in \mathcal{H}.
\end{align*}
Let $C_n = \mathcal{N}_{\infty}(\lambda)/n+ \lambda^{\beta}$. It will be shown that
\begin{align}
\label{eq:objlimit}
\lim_{D \to \infty} \liminf_{n \to \infty} \Pr\left(\inf_{\|f\|_{\mathcal{H},\lambda}=D} L_n(f_{\lambda}+C_n^{1/2}f)> L_n(f_\lambda) \right) = 1.
\end{align}
The convexity of $\rho$ from assumption \ref{A2} and the strict convexity $f \mapsto \|f\|_{\mathcal{H}}^2$ in combination with \eqref{eq:objlimit} guarantee the existence of a unique minimizer $\widehat{f}_n$ such that $\|\widehat{f}_n-f_\lambda\|_{\mathcal{H},\lambda} = O_{\mathbb{P}}(C_n^{1/2})$ , see, e.g., \citep[Lemma 1]{Kal:2022}. By definition of the norm $\|\cdot\|_{\mathcal{H},\lambda}$, this implies
\begin{align*}
\|\widehat{f}_n-f_\lambda\|_{\mathcal{L}^2_v(X)} \leq \|\widehat{f}_n-f_\lambda\|_{\mathcal{H},\lambda} = O_{\mathbb{P}}(C_n^{1/2}).
\end{align*}
Minkowski's inequality now gives
\begin{align*}
\|\widehat{f}_n-f_0\|_{\mathcal{L}^2_v(X)}\leq \|\widehat{f}_n-f_\lambda\|_{\mathcal{L}^2_v(X)} + \|f_\lambda- f_0 \|_{\mathcal{L}^2_v(X)} \leq  O_{\mathbb{P}}(C_n^{1/2}) + O(\lambda^{\beta/2}) = O_{\mathbb{P}}(C_n^{1/2}),
\end{align*}
where the last equality is due to the definition of $C_n$. From this it follows that in order to prove the result, we need to establish \eqref{eq:objlimit}. We establish this critical result by applying \ref{A3} and the fundamental theorem of calculus for Lebesgue integrals to decompose
\begin{align*}
L_n(f_{\lambda}	+C_n^{1/2}f) - L_n(f_{\lambda}) & = \frac{1}{n} \sum_{i=1}^n \int_{R_i}^{R_i-C_n^{1/2}f(x_i)} \left\{ \psi\left(\epsilon_i+t \right) - \psi \left( \epsilon_i \right) \right\} d t 
\\ & \quad - \frac{C_n^{1/2}}{n} \sum_{i=1}^n \psi\left(\epsilon_i \right) f(x_i) + \lambda C_n \|f\|_{\mathcal{H}}^2 +  2 \lambda C_n^{1/2} \langle f, f_\lambda \rangle_{\mathcal{H}}
\\ & = I_1(f) + I_2(f) + I_3(f),
\end{align*}
where
\begin{align*}
I_1(f) &:= \frac{1}{n} \sum_{i=1}^n \int_{R_i}^{R_i-C_n^{1/2}f(x_i)} \left\{ \psi\left(\epsilon_i+t \right) - \psi \left( \epsilon_i \right) \right\} d t  + \lambda C_n \|f\|_{\mathcal{H}}^2 \\
I_2(f) &: = -\frac{C_n^{1/2}}{n} \sum_{i=1}^n \psi(\epsilon_i) f(x_i) \\
I_3(f) &:= 2 \lambda C_n^{1/2} \langle f, f_\lambda \rangle_{\mathcal{H}}. 
\end{align*}
We will show that there exists a strictly positive $c_0$ with the property that
\begin{align}
\inf_{\|f\|_{\mathcal{H},\lambda}=D} \mathbb{E}[I_1(f)] & \geq c_0 D^2 C_n+ O(1)  DC_n  \label{eq:I1} \\
\sup_{\|f\|_{\mathcal{H},\lambda} \leq D} |I_1(f) - \mathbb{E}[I_1(f)]| & = O_{\mathbb{P}}(1) D C_n  \label{eq:I1md}\\
\sup_{\|f\|_{\mathcal{H},\lambda} \leq D} |I_2(f)| & = O_{\mathbb{P}}(1)D C_n  \label{eq:I2}
\\ \sup_{\|f\|_{\mathcal{H},\lambda} \leq D} |I_3(f)| & = O(1) D C_n \label{eq:I3}. 
\end{align}
If \eqref{eq:I1}--\eqref{eq:I3} were to hold, then \eqref{eq:objlimit} would follow, as for every $\epsilon>0$ we would be able to find a large enough $D = D_{\epsilon} > 1$ (so that the quadratic term dominates) and let $n \to \infty$ to obtain
\begin{align*}
\liminf_{n \to \infty} \Pr\left(\inf_{\|f\|_{\mathcal{H},\lambda}=D} L_n(f_{\lambda}+C_n^{1/2}f)> L_n(f_\lambda) \right) \geq 1-\epsilon.
\end{align*}
As this holds for every $\epsilon>0$, the double limit in \eqref{eq:objlimit} may  be deduced. For improved clarity, assertions  \eqref{eq:I1}--\eqref{eq:I3} are proven separately in Lemmas \ref{lem:I1}--\ref{lem:I3} below.
\end{proof}

\begin{lemma}
\label{lem:I1}
Under the conditions of Theorem~\ref{thm:1}, there exists a strictly positive $c_0$ such that
\begin{align*}
\inf_{\|f\|_{\mathcal{H},\lambda}=D}\mathbb{E}[I_1(f)] & \geq c_0 D^2 C_n+ O(1) D C_n,
\end{align*}
as $n \to \infty$, i.e., \eqref{eq:I1} holds.
\end{lemma}

\begin{proof}

Define the $D$-ball in $\mathcal{H}$: $\mathcal{B}_{D} := \{f \in \mathcal{H}: \|f\|_{\mathcal{H},\lambda} \leq D\}$. We begin by observing that, by Proposition~\ref{prop:2}, the reproducing kernel $K_{\lambda}$ satisfies
\begin{align*}
\sup_{x \in X}\|K_{\lambda}(x, \cdot)\|_{\mathcal{H},\lambda}^2 = \sup_{x \in X} K_{\lambda}(x,x)  = \mathcal{N}_{\infty}(\lambda).
\end{align*}
Our limit assumptions ensure that $\mathcal{N}_{\infty}(\lambda) C_n = \mathcal{N}_{\infty}(\lambda)( \mathcal{N}_{\infty}(\lambda)/n + \lambda^{\beta}) \to 0$, as $n \to \infty$. Hence, for every $f \in \mathcal{B}_D$, using the reproducing property and the Schwarz inequality,
\begin{align}
\label{eq:emb}
C_n^{1/2} \max_{1\leq i \leq n} |f(x_i)| \leq D C_n^{1/2} \sup_{x \in X} \|K_{\lambda}(x,\cdot)\|_{\mathcal{H}, \lambda} = D C_n^{1/2} \sqrt{\mathcal{N}_{\infty}(\lambda)} \to 0,
\end{align}
uniformly in $\mathcal{B}_D$. At the same time, appealing to Lemma~\ref{lem:intprop}, we have
\begin{align}
\label{eq:rem}
\max_{1 \leq i \leq n} |R_i| \leq \sup_{x \in X} |f_0(x) - f_{\lambda}(x)| \leq c_0 \lambda^{\beta-\alpha} \to 0,
\end{align}
since $\alpha<\beta$ and $\lambda \to 0$. From \eqref{eq:emb} and \eqref{eq:rem} it follows that the domain of integration shrinks to zero. Taking expectations with respect to the errors $\{\epsilon_i\}_{i=1}^n$, that is, conditionally on $\{x_i\}_{i=1}^n$, and using their independence contained in assumption \ref{A4}, we find
\begin{align*}
\mathbb{E}_{\epsilon} \left[ \frac{1}{n} \sum_{i=1}^n \int_{R_i}^{R_i-C_n^{1/2}f(x_i)} \left\{ \psi\left(\epsilon_i+t \right) - \psi \left( \epsilon_i \right) \right\} d t \right] & = \frac{1}{n} \sum_{i=1}^n \int_{R_i}^{R_i-C_n^{1/2}f(x_i)} \mathbb{E}[\left\{ \psi\left(\epsilon_1+t \right) - \psi \left( \epsilon_1 \right) \right\}] d t 
\\ & = \frac{1}{n} \sum_{i=1}^n \int_{R_i}^{R_i-C_n^{1/2}f(x_i)} \{\gamma t + o(t) \} d t,
\end{align*}
where to obtain this equality we have used \ref{A4} and the shrinking limits of integration. Evaluating this integral, we find that this may be bounded from below by
\begin{align*}
\frac{\gamma C_n }{2 n} \sum_{i=1}^n |f(x_i)|^2(1+o(1))  - \frac{\gamma C_n^{1/2}}{n} \sum_{i=1}^n |f(x_i)| |R_i| (1+o(1)).
\end{align*}
The $o(1)$-term is non-random and does not depend on $\{x_i\}_{i=1}^n$, by the uniformity in \eqref{eq:emb} and \eqref{eq:rem}. Taking expectations with respect to $\{x_i\}_{i=1}^n$ using the law of iterated expectations, we obtain
\begin{align*}
\mathbb{E}[I_1(f)] & \geq \frac{\gamma}{2} C_n \int_{X} |f(t)|^2 v(d t) (1+o(1)) - \gamma C_n^{1/2} \int_{X} |f(t)| |f_{\lambda}(t)-f_0(t)| v(d t) (1+o(1)) + \lambda C_n \|f\|_{\mathcal{H}}^2.
\end{align*}
Applying the Schwarz inequality on the second integral along with Lemma~\ref{lem:intprop}, we find $\int_{X} |f(t)| |f_{\lambda}(t)-f_0(t)| v(dt) \leq \|f\|_{\mathcal{L}^2_v(X)} \|f_{\lambda} - f_0 \|_{\mathcal{L}^2_v(X)} \leq C_n^{1/2} \|f\|_{\mathcal{L}^2_v(X)}$, as $\|f_{\lambda} - f_0 \|_{\mathcal{L}^2_v(X)} = O(\lambda^{\beta/2}) = O(C_n^{1/2})$. Moreover,  for all large $n$, $1/2 \leq 1+o(1) \leq 3/2$, hence
\begin{align*}
\mathbb{E}[I_1(f)] & \geq  \frac{\gamma}{4} C_n \|f\|_{\mathcal{L}^2_v(X)}^2 -  \frac{3\gamma}{2} C_n \|f\|_{\mathcal{L}^2_v(X)} + \lambda C_n \|f\|_{\mathcal{H}}^2 \geq \min\left(\frac{\gamma}{4}, 1 \right) C_n \|f\|_{\mathcal{H},\lambda}^2 - \frac{3 \gamma}{2} C_n \|f\|_{\mathcal{H},\lambda},
\end{align*}
because, by definition, $\|f\|_{\mathcal{H},\lambda}^2 = \|f\|_{\mathcal{L}^2_v(X)}^2 + \lambda \|f\|_{\mathcal{H}}^2 \geq \|f\|_{\mathcal{L}^2_v(X)}^2$. Taking the infimum, 
\begin{align*}
\inf_{\|f\|_{\mathcal{H},\lambda} = D}\mathbb{E}[I_1(f)] \geq c_0 C_n D^2 -c_1  C_n D,
\end{align*}
for all large $n$ with $c_0 = \min(\gamma/4, 1)$ and $c_1 = 3\gamma/2$. The proof is complete.

\end{proof}

\begin{lemma}
\label{lem:I1md}
Under the conditions of Theorem~\ref{thm:1}, 
\begin{align*}
\sup_{\|f\|_{\mathcal{H},\lambda} \leq D} |I_1(f) - \mathbb{E}[I_1(f)]| & = O_{\mathbb{P}}(1) D C_n,
\end{align*}
i.e., \eqref{eq:I1md} holds.
\end{lemma}

\begin{proof}
The proof is by symmetrization and contraction. Write
\begin{align*}
I_1(f) - \mathbb{E}[I_1(f)] & = \frac{1}{n} \sum_{i=1}^n \left\{  \int_{R_i}^{R_i-C_n^{1/2}f(x_i)} \left\{ \psi\left(\epsilon_i+t \right) - \psi \left( \epsilon_i \right) \right\} d t - \right. 
\\ & \left.  \quad -  \mathbb{E} \left[ \int_{R_i}^{R_i-C_n^{1/2}f(x_i)}  \left\{ \psi\left(\epsilon_i+t \right) - \psi \left( \epsilon_i \right) \right\} d t \right] \right\}
\\ & = \frac{1}{n} \sum_{i=1}^n \left\{ U_i(f(x_i)) - \mathbb{E}[U_i(f(x_i))]\right\},
\end{align*}
with $U_i(f(x_i)) = \int_{R_i}^{R_i-C_n^{1/2}f(x_i)} \{\psi(\epsilon_i+t)-\psi(\epsilon_i)\} d t$. These processes are independent and identically distributed for each given $n$. By symmetrization \citep[Lemma 2.3.1]{VDV:1996},
\begin{align*}
\mathbb{E}\left[ \sup_{\|f\|_{\mathcal{H}, \lambda} \leq D} \left|I_1(f) - \mathbb{E}[I_1(f)] \right|  \right] \leq  2 \mathbb{E}\left[\sup_{\|f\|_{\mathcal{H}, \lambda}  \leq D}\left| \frac{1}{n} \sum_{i=1}^n \xi_i U_{i}(f(x_i)) \right| \right],
\end{align*}
where $\{\xi_i\}_{i=1}^n$ are i.i.d. Rademacher random variables taking values $\{-1,1\}$ with probability $1/2$ each and are also independent of $(x_i, \epsilon_i)$. It is clear that $U_i(0) = 0$ for all $i=1, \ldots, n$. We show that the function $f \mapsto U_i(f(x_i))$ is Lipschitz with Lipschitz constant $M_1 C_n^{1/2}$. Indeed, recalling that, by the proof of Lemma~\ref{lem:I1}, in particular, \eqref{eq:emb} and \eqref{eq:rem}, the domain of integration shrinks to zero, assumption \ref{A3} may be applied to obtain
\begin{align*}
\left|U_i(f_1(x_i)) - U_i(f_2(x_i)) \right| & = \left| \int_{R_i-C_n^{1/2} f_1(x_i)}^{R_i-C_n^{1/2}f_2(x_i)} \{\psi(\epsilon_i+t) - \psi(\epsilon_i)\}  d t \right|  \leq  M_1  C_n^{1/2}|f_1(x_i)-f_2(x_i)|,
\end{align*}
as claimed. Besides, by the definition of $\mathcal{N}_{\infty}(\lambda)$ and homogeneity, $\sup_{\|f\|_{\mathcal{H},\lambda}\leq D} \|f\|_{\infty} = D \sqrt{\mathcal{N}_{\infty}(\lambda)} < \infty$. Conditionally on $\{(\epsilon_i,x_i)\}_{i=1}^n$, the conditions of Theorem 4.12 of \citep{Tal:1991} are satisfied and so, by the contraction inequality given there,  valid for all $n \in \mathbbm{N}$,
\begin{align*}
\mathbb{E}\left[\sup_{\|f\|_{\mathcal{H}, \lambda} \leq D}\left| \frac{1}{n} \sum_{i=1}^n \xi_i U_{i}(f(x_i)) \right| \right] \leq M_1 C_n^{1/2} \mathbb{E}\left[\sup_{\|f\|_{\mathcal{H}, \lambda} \leq D}\left| \frac{1}{n} \sum_{i=1}^n \xi_i f(x_i) \right| \right]. 
\end{align*}
To complete the proof we bound the expectation appearing on the right side of this inequality. 

By the reproducing property, the linearity of the inner product and the definition of the dual norm on $\mathcal{H}$, we have
\begin{align*}
\mathbb{E}\left[\sup_{\|f\|_{\mathcal{H}, \lambda} \leq D}\left| \frac{1}{n} \sum_{i=1}^n \xi_i f(x_i) \right| \right] & = \mathbb{E}\left[\sup_{\|f\|_{\mathcal{H}, \lambda} \leq D}\left| \frac{1}{n} \sum_{i=1}^n \xi_i \langle f, K_{\lambda}(x_i,\cdot) \rangle_{\mathcal{H},\lambda}\right| \right]
\\ & =  D \mathbb{E}\left[ \left\| \frac{1}{n} \sum_{i=1}^n \xi_i K_{\lambda}(x_i, \cdot) \right\|_{\mathcal{H},\lambda} \right].
\end{align*}
Using the inequality $\mathbb{E}[|X|] \leq \sqrt{\mathbb{E}[X^2]}$ and the independence of $\{\xi_i\}_{i=1}^n$ from $\{x_i\}_{i=1}^n$, we obtain
\begin{align*}
\mathbb{E}\left[\sup_{\|f\|_{\mathcal{H}, \lambda} \leq D}\left| \frac{1}{n} \sum_{i=1}^n \xi_i f(x_i) \right| \right] & \leq D \sqrt{ \mathbb{E}\left[ \left\| \frac{1}{n} \sum_{i=1}^n \xi_i K_{\lambda}(x_i, \cdot) \right\|_{\mathcal{H},\lambda}^2 \right] }=  \frac{D}{n}  \sqrt{ \sum_{i=1}^n \sum_{j=1}^n \mathbb{E}[K_{\lambda}(x_i,x_j)] \mathbb{E}\left[\xi_i \xi_j \right] }
\\ & = \frac{D}{\sqrt{n}} \sqrt{ \int_{X} K_{\lambda}(x,x) v(dx)  } \leq \frac{D}{\sqrt{n}} \sqrt{ v(X) \sup_{x \in X} K_{\lambda}(x,x) } = D O\left( \sqrt{ \frac{\mathcal{N}_{\infty}(\lambda)}{n} }\right).
\end{align*}
But, by definition of $C_n$, $\mathcal{N}_{\infty}(\lambda)/n \leq C_n$, therefore
\begin{align*}
\mathbb{E}\left[ \sup_{\|f\|_{\mathcal{H},\lambda} \leq D} |I_1(f) - \mathbb{E}[I_1(f)]| \right] = O(1) D C_n,
\end{align*}
and the proof is completed by an appeal to Markov's inequality.
\end{proof}

\begin{lemma}
\label{lem:I2}
Under the conditions of Theorem~\ref{thm:1},
\begin{align*}
\sup_{\|f\|_{\mathcal{H},\lambda} \leq D} |I_2(f)| & = O_{\mathbb{P}}(1)D C_n,
\end{align*}
i.e., \eqref{eq:I2} holds.
\end{lemma}
\begin{proof}

Using the reproducing property and the definition of the dual norm, we find
\begin{align*}
\sup_{\|f\|_{\mathcal{H},\lambda} \leq D} \left| \frac{1}{n} \sum_{i=1}^n \psi \left(\epsilon_i \right) f(x_i) \right| = D \left\| \frac{1}{n} \sum_{i=1}^n \psi\left( \epsilon_i  \right)	K_{\lambda}(x_i,\cdot) \right\|_{\mathcal{H},\lambda}.
\end{align*}
Squaring and taking expectations recalling that, by assumption \ref{A4}, the $\{\epsilon_i\}_{i=1}^n$ are i.i.d., independent of the $\{x_i\}_{i=1}^n$, $\mathbb{E}[\psi(\epsilon_1)]=0$ and $\mathbb{E}[|\psi(\epsilon_1)|^2]<\infty$, we have
\begin{align*}
\mathbb{E} \left[ \left\| \frac{1}{n} \sum_{i=1}^n \psi\left( \epsilon_i  \right)	K_{\lambda}(x_i,\cdot) \right\|_{\mathcal{H},\lambda}^2 \right] & = \frac{1}{n^2} \sum_{i=1}^n \sum_{j=1}^n \mathbb{E}[K_{\lambda}(x_i,x_j)] \mathbb{E}[ \psi(\epsilon_i)  \psi(\epsilon_j)]   = \frac{\mathbb{E}[ |\psi(\epsilon_1)|^2]}{n}  \mathbb{E}\left[ K_{\lambda}(x_1,x_1) \right] \\& \leq c_0 \frac{\sup_{x \in X} K_{\lambda}(x,x)}{n} = c_0 \frac{\mathcal{N}_{\infty}(\lambda)}{n} =O(C_n),
\end{align*}
by Proposition~\ref{prop:2}. The result follows by Markov's inequality.
\end{proof}

\begin{lemma}
\label{lem:I3}
Under the conditions of Theorem~\ref{thm:1},
\begin{align*}
\sup_{\|f\|_{\mathcal{H},\lambda} \leq D} |I_3(f)| & = O(1) D C_n,
\end{align*}
i.e., \eqref{eq:I3} holds.
\end{lemma}
\begin{proof}
By definition of the norm $\|\cdot\|_{\mathcal{H},\lambda}$, $\lambda^{1/2} \|f\|_{\mathcal{H}} \leq \|f\|_{\mathcal{H},\lambda}$, hence, by the Schwarz inequality in $\mathcal{H}$,
\begin{align*}
\sup_{ \|f\|_{\mathcal{H},\lambda} \leq D} |I_3(f)| & \leq 2 \lambda^{1/2} \|f_{\lambda}\|_{\mathcal{H}} C_n^{1/2} \sup_{ \|f\|_{\mathcal{H},\lambda} \leq D} \lambda^{1/2} \|f\|_{\mathcal{H}} \leq 2 \lambda^{1/2} \|f_\lambda\|_{\mathcal{H}} C_n^{1/2} D.
\end{align*}
By Lemma~\ref{lem:intprop}, $\|f_{\lambda}\|_{\mathcal{H}} =O(\lambda^{-(1-\beta)/2})$, so
\begin{align*}
\sup_{ \|f\|_{\mathcal{H},\lambda} \leq D} |I_3(f)| = D O( \lambda^{1/2} C_n^{1/2} \lambda^{-(1-\beta)/2}) = D O( \lambda^{\beta/2} C_n^{1/2}) = D C_n O(1),
\end{align*}
given that $\lambda^{\beta} \leq C_n$; the proof is complete.
\end{proof}

\begin{proof}[Proof of Theorem~\ref{thm:2}]

Most of the proof of Theorem~\ref{thm:1} goes through after an estimate for $\mathcal{N}_{\infty}(\lambda)$ as $\lambda \to 0$ is obtained. For this, we recall that there exists a finite $M>0$ such that $\sup_{j} \|\phi_j\|_{\infty} \leq M$  and, by properties of the univariate Sobolev space $\mathcal{H}^m([0,1])$, $\gamma_j \asymp j^{2m}$ for all large $j$ \citep[see, e.g.,][Theorem 2.8.3]{Hsing:2015}. Consequently, using \eqref{eq:tRK} and Proposition~\ref{prop:2} and taking advantage of the product structure, we find
\begin{align*}
\mathcal{N}_{\infty}(\lambda) \leq M^{2d} \sum_{i_1, \ldots, i_d} \frac{1}{1+ \prod_{j=1}^d \gamma_{i_j}} \leq dM^{2d} + M^{2d} \int_{1}^{\infty} \ldots \int_{1}^{\infty} \frac{d x_1 \ldots d x_d}{1+\lambda x_1^{2m} \ldots x_d^{2m}},
\end{align*}
with the second inequality following from integral approximation on each coordinate. Change variables according to $z_1 = x_1, z_2 = x_1 x_2, \ldots, z_d = \prod_{j=1}^d x_j$ mapping $(1,\infty)^d$ bijectively to $\{1< z_1 < z_2 < \ldots <z_d < \infty\}$. The absolute Jacobian of this transformation is $\prod_{j=1}^{d-1} z_j^{-1}$, so 
\begin{align*}
\int_{1}^{\infty} \ldots \int_{1}^{\infty} \frac{d x_1 \ldots d x_d}{1+\lambda x_1^{2m} \ldots x_d^{2m}} & = \int_{1}^{\infty} \ldots \int_{1}^{z_2}  \frac{dz_1 \ldots dz_d}{z_1 \ldots z_{d-1}(1 + \lambda z_d^{2m})} = \frac{1}{(d-1)!} \int_{1}^{\infty} \frac{[\log(z_d)]^{d-1}}{1+\lambda z_d^{2m}} dz_d
\\ & \leq \frac{\lambda^{-1/(2m)}}{(d-1)!} \int_{0}^{\infty} \frac{[(2m)^{-1}\log\left( \frac{1}{\lambda}  \right) + \log(w)]^{d-1}}{1+w^{2m}} dw 
\\& = O \left( \frac{1}{\lambda^{\frac{1}{2m}}} \left[\log\left(\frac{1}{\lambda} \right) \right]^{d-1} \right),
\end{align*}
where the second equality is obtained from $\int_{1}^{z_d} \ldots \int_{1}^{z_2} \prod_{i=1}^{d-1} z_i^{-1} dz_1 \ldots dz_{d-1} = [\log(z_d)]^{d-1}/(d-1)!$, which may be proved by induction on $d \geq 1$. The inequality results from changing variables according to $w = \lambda^{1/(2m)} z_d$ and extending the lower limit of integration to $0$. To pass to the last equality we have used $\lambda \to 0$ and the observation that $\int_{0}^{\infty} |[\log(w)]^{d-1}/(1+w^{2m})| dw < \infty $ for all $d, m \geq 1$. Thus $\mathcal{N}_{\infty}(\lambda) = O(\lambda^{-1/(2m)} [\log(1/\lambda) ]^{d-1})$, as $\lambda \to 0$. 

Lemma~\ref{lem:intprop} remains valid in the present setting, as do Lemmas~\ref{lem:I1md}--\ref{lem:I3}. Moreover, since, by assumption, $v$ has Lebesgue density $h$ bounded away from zero and infinity on $[0,1]^d$, repeating the steps in the proof of Lemma~\ref{lem:I1}, it may be verified that there exist positive constants $c_1$ and $c_2$ such that 
\begin{align*}
\mathbb{E}[I_1(f)] \geq c_1\frac{\gamma}{4} C_n \|f\|_{\mathcal{L}^2([0,1]^d)}^2 - c_2\frac{3\gamma}{2} C_n \|f\|_{\mathcal{L}^2([0,1]^d)} + \lambda C_n \|f\|_{\mathcal{H}^m([0,1])^{\otimes d}}^2,
\end{align*}
for all large $n$. This implies that the infimum over the $D$-sphere in $\|f\|_{\mathcal{H}^m([0,1])^{\otimes d}}^2$ is lower bounded by the product of $C_n$ and a quadratic function of $D$ with a strictly positive coefficient on $D^2$. Thus,
\begin{align*}
\lim_{D \to \infty} \liminf_{n \to \infty} \Pr\left(\inf_{\| f\|_{\mathcal{H}^m([0,1])^{\otimes d}, \lambda}=D} L_n(f_{\lambda}+C_n^{1/2}f)> L_n(f_\lambda) \right) = 1,
\end{align*}
completing the proof.

\end{proof}

\end{document}